\documentclass[a4paper,12pt]{article}
\usepackage{amssymb}
\usepackage{epsfig}
\usepackage{amsfonts}
\usepackage{amsmath}
\usepackage{euscript}
\usepackage{amscd}
\usepackage{amsthm}
\DeclareMathAlphabet{\mathpzc}{OT1}{pzc}{m}{it}
\newtheorem{thm}{Theorem}[section]
\newtheorem{lem}[thm]{Lemma}
\newtheorem{prop}[thm]{Proposition} 
\newtheorem{defn}[thm]{Definition}
\newtheorem{cor}[thm]{Corollary}
\newtheorem{rem}[thm]{Remark}
\newtheorem{ex}[thm]{Example}
\newtheorem{que}[thm]{Question}

\newcommand{\bQ}{\mathbb Q}
\newcommand{\bN}{\mathbb N}
\newcommand{\bC}{\mathbb C}

\author{Nikhilesh Dasgupta$^*$ and Sergey Gaifullin$^\dagger$ \\
{\small {\it $^*$Department of Mathematics,}}\\
{\small {\it Gandhi Institute of Technology and Management (GITAM) University,}} \\
{\small {\it Bengaluru, Karnataka 561203, India.}}\\
{\small {\it e-mail: its.nikhilesh@gmail.com, ndasgupt@gitam.edu}}\\
{\small {\it $^\dagger$Faculty of Computer Science,}}\\
{\small {\it HSE University,}} \\
{\small {\it Pokrovsky Boulevard 11, Moscow 109028, Russia.}}\\
{\small {\it e-mail: sgayf@yandex.ru}}}
\title{On locally nilpotent derivations of polynomial algebra in three variables}

\begin{document}
\date{}
\maketitle

\begin{abstract}
In this paper we investigate locally nilpotent derivations on the polynomial algebra in three variables over a field of characteristic zero. We introduce an iterating construction giving all locally nilpotent derivations of rank $2$. This construction allows to get examples of non-triangularizable locally nilpotent derivations of rank $2$. We also show that the well-known example of a locally nilpotent derivation of rank $3$, given by Freudenburg, is a member of a large family of new examples of rank $3$ locally nilpotent derivations. Our approach is based on considering all locally nilpotent derivations commuting with a given one. We obtain a characterization of locally nilpotent derivations with a given rank in terms of sets of commuting locally nilpotent derivations.

\smallskip
\noindent
{\small {{\bf Keywords}. Polynomial ring, locally nilpotent derivation, rank, kernel, triangularizable derivation.}

\smallskip
\noindent
{\small {{\bf 2020 MSC}. Primary: 14R10, 14R20; Secondary: 13A50, 32M17.}}
}

\end{abstract}

\section{Introduction}
By a ring, we mean a Noetherian ring with unity. Let $k$ be a field of characteristic zero and $B$ be a $k$-algebra. We use the notation $B=k^{[n]}$ to denote that $B$ is isomorphic to the polynomial algebra in $n$ variables over $k$. A $k$-linear map $\delta : B \rightarrow B$ is said to be a $k$-{\it derivation} if it satisfies the Leibnitz rule: $\delta(ab)=a\delta(b)+b\delta(a)$ for all $a,b \in B$. A $k$-derivation $\delta$ is a said to be a {\it locally nilpotent derivation} (abbrev. LND) if, for each $b \in B$, there exits $n(b) \in \bN$ such that $\delta^{n(b)}(b)=0$. The $B$-subalgebra $A:=\{ b \in B~\vert ~ \delta(b)=0\}$ is called the {\it kernel} of $\delta$ and denoted by ${\rm Ker}~\delta$. The set of all LNDs of $B$ is denoted by ${\rm LND}(B)$.

When $k$ is algebraically closed, the group of regular automorphisms $\mathrm{Aut}(k^{[n]})$ is a classical object of investigation in affine algebraic geometry.  There are a lot of famous long standing problems about $\mathrm{Aut}(k^{[n]})$ such as Jacobian conjecture, Linearization problem, Tameness problem and so on. One of the approaches to study $\mathrm{Aut}(k^{[n]})$  is to study all algebraic subgroups in it. Each connected linear algebraic group is generated by one-dimensional subgroups isomorphic to the additive or the multiplicative group of $k$. We call such subgroups $\mathbb{G}_a$ and $\mathbb{G}_m$-subgroups, respectively. The $\mathbb{G}_a$-subgroups correspond to locally nilpotent derivations  of $k^{[n]}$.

 LNDs  became a powerful tool for investigating $k$-algebras.  In 1996, Makar-Limanov~\cite{ML} introduced a new invariant of a $k$-algebra $B$, which was later named the {\it Makar-Limanov invariant}. It is the intersection of kernels of all LNDs on $B$ and is denoted by ${\rm ML}(B)$. This invariant allows one to prove that the  Koras-Russell cubic $\{x+x^2y+z^2+t^3=0\}$ is not isomorphic to the three dimensional affine space. Generalization of this result given by Kaliman and Makar-Limanov  in 1997 was the final step in the proof of the Linearization problem for $\mathbb{C}^*$-action on $\mathbb{C}^3$ \cite{KML}. One can also use ${\rm ML}(B)$ for the purpose of describing all automorphisms of a $k$-algebra $B$ \cite{ML2, MJ,P,Gai}.  In~\cite{DaGu},  characterizations of $k^{[2]}$ and $k^{[3]}$ in terms of the Makar-Limanov invariant are obtained.

Another popular concept based on the concept of LNDs is that of flexibility of an affine algebraic variety. Let $\mathrm{SAut}(X)$ denote the group generated by all $\mathbb{G}_a$-subgroups in $\mathrm{Aut}(k[X])$, where $k[X]$ is the algebra of regular functions on an affine algebraic variety $X$. In \cite{AFKKZ} it is proved that if $\mathrm{SAut}(X)$ acts on the regular locus $X^{reg}$ transitively, then it acts on  $X^{reg}$ infinitely transitively, i.e. every  $m$-tuple $\{a_1,\ldots, a_m\}$ of distinct points in $X^{reg}$ can be moved to any $m$-tuple $\{b_1,\ldots, b_m\}$ of distinct points in $X^{reg}$ by an element of $\mathrm{SAut}(X)$. Such varieties are called {\it flexible}. It turns out that many natural classes of varieties including affine spaces are flexible \cite{AKZ, AFKKZ, GSh, Pe}.

Although LNDs are popular objects of investigation, even for many well-known algebras it is a complicated problem to describe all LNDs on these algebras. For example, in case of the polynomial algebra $k^{[n]}$, there is no complete description starting from $n=3$. For $n=1$, each LND of $k[x]$ has the form $c\frac{\partial}{\partial x}$, where $c\in k$.  For $n=2$, each LND of $k[x,y]$ has the form $f(u)\frac{\partial}{\partial v}$, where $\{u,v\}$ is a coordinate system of $k^{[2]}$, i.e. $k[x,y]=k[u,v]$ (Theorem \ref{R}). When $n\geqslant 3$, there is no such explicit description of an arbitrary LND, but there are many known examples. Indeed, each {\it triangularizable} derivation is an LND. By a triangularizable derivation, we mean a derivation $\delta$ on $k^{[n]}$ such that in some coordinate system $\{x_1,\ldots, x_n\}$ we have $\delta (x_i)\in k[x_1,\ldots, x_{i-1}]$, in particular, $\delta(x_1)\in k$.  For $n=3$, there exists a complete description of LNDs which contain at least one variable in its kernel, given by Daigle and Freudenburg (Theorem \ref{df}).  It seems that starting from $n=4$ the situation becomes very difficult. An important technique to study an LND is by investigating its kernel. Indeed, it is easy to see  that for $n\leqslant 2$ the kernel of any non-zero LND of $k^{[n]}$ is a polynomial ring. For $n=3$, this fact is the statement of Miyanishi's theorem \cite[Theorem~5.1]{F}. But for $n\geqslant 5$ the kernel of an LND can be even non finitely generated~\cite{DF}. For $n=4$, though it is known that the kernel of an LND is finitely generated if it annihilates a variable \cite[Theorem 1]{BhD} and is also a polynomial ring if in addition, the kernel is regular \cite{BGL}, the question of finite generation of the kernel remains open in general.  That is why our paper is devoted to the case $n=3$. This paper is an attempt to understand how do LNDs on $k[x,y,z]$ look like.

If $\delta$ is an LND and $h\in \mathrm{Ker}\,\delta$, then $h\delta$ is also an LND, which is called a {\it replica} of $\delta$. So, to investigate all LNDs in some sense it is sufficient to investigate irreducible ones, i.e. LNDs whose image is not contained in any non-trivial principal ideal. Indeed, all LNDs are replicas of irreducible ones. Each derivation of $k[x_1,\ldots, x_n]$ has the form 
$$
\delta=f_1\frac{\partial}{\partial x_1}+\ldots+f_n\frac{\partial}{\partial x_n}, \qquad \text{where} \quad f_i\in k[x_1,\ldots, x_n].
$$
If we change the system of coordinates, i.e. consider such $y_1,\ldots, y_n$ that $k[y_1,\ldots, y_n]=k[x_1,\ldots, x_n]$, then $\delta$ has the form 
$$\widetilde{f}_1\frac{\partial}{\partial y_1}+\ldots+\widetilde{f}_n\frac{\partial}{\partial y_n}, \qquad \widetilde{f}_i\in k[y_1,\ldots, y_n].$$
The minimal number of non-zero $\widetilde{f}_1,\ldots, \widetilde{f}_n$ for various systems of coordinates is called {\it rank} of the derivation $\delta$, denoted by $\mathrm{rk}~\delta$.  If $\mathrm{rk}\,\delta =0$, then $\delta=0$. Every LND of rank $1$  is a replica of the derivation $\frac{\partial}{\partial x_1}$ in some coordinate system $\{x_1, \dots ,x_n\}$ of $k^{[n]}$ (Proposition \ref{fr95c}). It is easy to prove that a replica of a triangularizable derivation has rank at most $n-1$. But the converse is not true even for $n=3$. On $k^{[3]}$, there are examples of non-triangularizable irreducible LNDs of rank 2. The first example of a non-triangularizable LND of $k^{[3]}$ was given by Bass \cite{B}, which was generalized to $k^{[n]}$ for $n \geqslant 4$ by Popov \cite{Po}. More examples were given by Freudenburg \cite[Example 2]{F95} and Daigle \cite[Example 3.5]{D96}. Later,  Freudenburg~\cite{F98} proved that there exist LNDs on $k^{[3]}$ of rank $3$. This example of Freudenburg (see Section \ref{rk3der}) is given by some rather complicated polynomials. The proof that this LND is of rank $3$, is based on the fact that it is homogeneous with respect to a $\mathbb{Z}$-grading on $B$. Later, he introduced the technique of local slice constructions \cite{Fr97} to generalize this example and obtained new examples of rank~$3$ LNDs. Similar examples were also obtained by Daigle, who used a geometric approach. The connection between his geometric approach and local slice constructions is described in \cite{Da07}. Although the local slice construction itself does not require any kind of
 homogeneity, the earlier examples of rank $3$ LNDs obtained by these methods were all homogeneous with respect to some positive $\mathbb{Z}$ gradings.

Our approach to studying LNDs of $k^{[3]}$ is based on investigating all LNDs that commute with a fixed one. It is easy to see that, given any LND $\delta$,  all LNDs equivalent to $\delta$ commute with $\delta$. Recall that, two LNDs are called {\it equivalent} if their kernels coincide. We prove the following characterization of irreducible LNDs of rank $3$: {\it an irreducible LND $\delta$ of $k^{[3]}$ has rank 3 if and only if every LND commuting with it is equivalent to $\delta$} (see Theorem~\ref{rk2thm1}). Note that, for a reducible LND, it is not true. We give an explicit criterion for an LND to have a non-equivalent commuting one (see Proposition \ref{rk2com1}). It allows us to construct LNDs even of rank $1$ without non-equivalent commuting ones. For example, one can consider the derivation $xy\dfrac{\partial}{\partial z}$ on $k[x,y,z]$. We also prove that each LND of $k^{[3]}$ can be restricted to a canonical (possibly, proper) subalgebra of $k^{[3]}$ isomorphic to $k^{[3]}$, such that the restriction is an LND of rank one. We prove that if a restriction of an irreducible LND has no non-equivalent commuting LNDs then the initial LND $\delta$ also has no non-equivalent commuting LNDs and hence has rank $3$ (Remark~\ref{nepsam} and Proposition~\ref{rrr}). This allows us to prove a sufficient condition when a local slice construction gives an LND of rank~3 (Corollary~\ref{lscor}). So, we obtain a new proof that Freudenburg's example is of rank $3$ and construct new non-homogeneous examples of rank $3$ LNDs (Theorem \ref{rk3final}). 

In this paper, we develop a new technique to obtain new LNDs via chains of commuting LNDs starting form basic partial derivations. We prove that every irreducible LND on $k^{[n]}$ of rank at most $2$ can be obtained by an iterative construction from partial derivations in some coordinate system (Corollary \ref{CLNDC3var}). We call this construction as the {\it Modified Commuting LND Construction} (abbrev. MC-construction). It has a close connection with amalgamated product structure of $\mathrm{Aut}(k^{[2]})$.  The minimal number of steps of this construction is defined to be the {\it level} of the LND. This gives a new numerical characteristic, which shows how complicated the LND is. The simplest LNDs of rank at most $2$ are LNDs of rank $1$, then come the triangularizable LNDs. As the level of an LND depends on the choice of coordinate system, one can take the minimum over a suitable collection of coordinate systems. That minimum number is called the {\it universal level} of $D$. We show that an LND has rank $1$ if and only if it has universal level $1$ and it is triangularizable if and only if it has universal level $3$ (Theorem \ref{crtr}). In our definition, no LND has universal level $2$. Using the MC-construction technique, we construct a new family of non-triangularizable LNDs of rank 2 (Example~\ref{ex1}). These examples are different from the earlier examples of Freudenburg \cite{F95} and Daigle \cite{D96}. In a forthcoming paper, we will give a description of the isotropy subgroup of an LND of $k^{[3]}$, i.e. the subgroup of ${\rm Aut}(k^{[3]})$ consisting of all automorphisms of $k^{[3]}$ commuting with this LND.

\section{Preliminaries}

For a ring $R$, we denote the multiplicative group of units by $R^*$. Let $k^{[n]}$ be the polynomial ring in $n$ variables over $k$. We say that $\mathbf{x}=\{x_1, \dots, x_n\}$, $x_i\in k^{[n]}$ is a {\it coordinate system} or {\it system of variables} if $k^{[n]}=k[x_1, \dots, x_n]$. By a {\it coordinate} or a {\it variable} we mean such an element that it can be included in a coordinate system. If $\alpha$ is an automorphism of $k^{[n]}$ such that $\alpha (x_i)=f_i$, then we use the notation $\alpha=(f_1,\ldots,f_n)$.

\begin{defn}
{\em For a coordinate system $\mathbf{x}=\{x_1, \dots, x_n\}$ of $k^{[n]}$, an automorphsim $\alpha=(f_1, \dots ,f_n)$ is said to be {\it triangular} if $f_i \in k[x_1, \dots ,x_i]$ for each $i$, $1 \leqslant i\leqslant  n$. We denote the subgroup of triangular automorphisms by ${\rm T}_n(k; \mathbf{x})$ or more simply by ${\rm T}_n(k)$ when the coordinate system $\mathbf{x}$ is understood from the context. The subgroup of ${\rm Aut}(k^{[n]})$ generated by ${\rm GL}_n(k)$ and the triangular automorphisms is called the {\it tame subgroup} relative to the coordinate system $\mathbf{x}$, and is denoted by ${\rm TA}_n(k;\mathbf{x})$.}
\end{defn}

It is well-known that every automorphism of $k^{[2]}$ is tame. This result was first proved by Jung \cite{Ju} for characteristic zero and generalized to arbitrary characteristic by Van der Kulk~\cite{VDK}. The next theorem gives the explicit structure of ${\rm Aut}(k^{[2]})$ \cite[Theorem~4.2]{F}. Though this result was initially contained in the works of Van der Kulk, it was Nagata who stated and proved it in this exact form \cite{Na}. His proof assumed $k$ to be algebraically closed. Later, in his Ph.D. thesis, Wright proved this theorem for the general case \cite{W1}.
\begin{thm}\label{jvk}
Let $k$ be any field. Then ${\rm Aut}(k^{[2]})$ has the structure of an amalgamated free product. More precisely,
$${\rm Aut}(k^{[2]})={\rm Af}_2(k) \ast_C {\rm T}_2(k),$$
where ${\rm Af}_2(k)$ is the subgroup of affine automorphisms
$$\rm{A f}_2(k)=\left\{
\begin{pmatrix}
x\\
y
\end{pmatrix}\mapsto
\begin{pmatrix}
a_{11}x+a_{12}y+b_1\\
a_{21}x+a_{22}y+b_2
\end{pmatrix} \Big \vert  \begin{pmatrix}
a_{11}&a_{12}\\
a_{21}&a_{22}
\end{pmatrix} \in {\rm GL}_2(k)~\text{and}~b_1,b_2 \in k\ 
 \right\}.$$ and $C := {\rm Af}_2(k) \cap {\rm T}_2(k)$.
\end{thm}

\smallskip
\noindent

Let us recall some basic facts about locally nilpotent derivations. A detailed treatise in this field can be found  in the book by Freudenburg~\cite{F}. Let $B$ be an affine $k$-domain. The set of all LNDs of $B$ we denote by ${\rm LND}(B)$. Each LND $D$ induces a degree function $\deg_{D}\colon B\setminus\{0\}\rightarrow\mathbb{Z}_{\geqslant 0}$, given by $$\deg_D\,f:=m,~\text{where}~D^m(f)\neq 0~\text{and}~D^{m+1}(f)=0.$$ This degree function satisfies the usual properties of a degree function: 
\begin{enumerate}
\item [\rm (i)] $\deg_D(f+g)\leqslant \mathrm{max}\{\deg_D\,f,\deg_D\,g\}$;
\item [\rm (ii)] $\deg_D(fg)=\deg_D\,f+\deg_D\,g$.
\end{enumerate}
For each non-zero LND $D$ its kernel $A=\mathrm{Ker}\,D$ is a factorially closed subalgebra such that the transcendence degree of $B$ over $A$ equals 1, see \cite[Principle 1]{F}. Recall that, a subalgebra $A$ of an algebra $B$ is said to be factorially closed in $B$, if for any $a,b \in B$, the condition $ab \in A$ implies that both $a$ and $b$ are in $A$. If $S \subset A \setminus \{0\}$ is multiplicatively closed, then $D$ induces an LND on $S^{-1}B$ with kernel $S^{-1}A$, see \cite[Principle~9]{F}.

\begin{defn}
{\em Let $B$ be a $k$-algebra. We say two locally nilpotent derivations $D_1$ and $D_2$ of $B$ to be {\it equivalent} if ${\rm Ker}~D_1={\rm Ker}~D_2$. 
}
\end{defn}

It is well-known that two non-zero LNDs $D_1$ and $D_2$ are equivalent if and only if there exist non-zero $a,b\in B$ such that $aD_1=bD_2$.
It is easy to see that both $a$ and $b$ can be chosen in ${\rm Ker}~D_1={\rm Ker}~D_2$, see \cite[Principle 12]{F}.

\begin{defn}
{\em
Let $D\in {\rm LND}(B)$. An element $s \in B$ is called a {\it slice} if $Ds=1$, and a 
{\it local slice} if $Ds \in {\rm Ker} D$ and $Ds \neq 0$, i.e., ${\rm deg}_D(s)=1$.
}
\end{defn}
A local slice exists for every non-zero LND. But the existence of a slice is a rather strict condition for an LND. The following important result is known as the Slice Theorem,  see \cite[Corollary 1.26]{F}.
\begin{thm}\label{st}
 Let $B$ a $k$-domain. Suppose $D \in {\rm LND}(B)$ admits a slice $s\in B$, and let $A={\rm Ker}~D$. Then 
$B=A[s]$ and $D=\frac{\partial}{\partial s}$.
\end{thm}

\smallskip
\noindent
The following assertion, see~\cite[Principle~11(d)]{F}, is a generalization of the above theorem. 
\begin{prop}\label{proplocsl}
Let $r$ be a local slice of $D$. Then the following holds $B\left[\frac{1}{D(r)}\right]=A\left[\frac{1}{D(r)}\right][r]$.
\end{prop}

Now we assume $B=k^{[n]}$. We start with the result, due to Rentschler \cite{Re}, which characterizes all locally nilpotent derivations on~$k^{[2]}$.
\begin{thm}\label{R}
Let $D$  be a non-zero LND of $k[x,y]$. There there exists a tame automorphism $\alpha \in TA_2(k; \{x,y\})$ and $p(x) \in k[x]$ such that $\alpha D  {\alpha}^{-1}=p(x)\dfrac{\partial }{\partial y}.$
\end{thm}

In particular, it follows from this theorem that the kernel of any non-zero LND of $k^{[2]}$ is isomorphic to $k^{[1]}$. 
The next theorem is a famous result of Miyanishi which gives the structure of the kernel of any non-zero LND on $k^{[3]}$, see \cite[Theorem 5.1]{F}.
\begin{thm}\label{miya}
Let $D (\neq 0 ) \in {\rm LND}(k^{[3]})$. Then ${\rm Ker}~D \cong k^{[2]}$.
\end{thm}

\begin{defn}{\em
The {\it rank} of $D$, denoted by $\mathrm{rk}\,D$, is defined to be the least integer $i$ for which
there exists a coordinate system $\{x_1,x_2,\dots,x_n\}$ of $B$ satisfying 
$$k[x_{i+1},\dots,x_n]\subseteq \mathrm{Ker}\,D.$$
}
\end{defn}
\noindent
It is easy to see that this definition is equivalent to the definition of rank given in the introduction. If $\mathrm{rk}\, D=0$, then $D=0$. So, for any non-zero LND  $D$, $1\leqslant \mathrm{rk}\,D\leqslant n$.

\smallskip
Let $B$ be an affine $k$-domain and $D\in {\rm LND}(B)$ with the kernel $A$ and the image $DB$. It can be proved that the intersection of the kernel and the image $A\cap DB$ is an ideal in $A$. It is called {\it plinth} ideal and is denoted by $\mathrm{pl}(D)$.
The following result called the Plinth Ideal Theorem follows from the works of Bonnet \cite{Bo} and Daigle-Kaliman \cite{DK}.
\begin{thm}\label{pl}
Let $B=k^{[3]}$, $D \in {\rm LND}(B)$ and $A={\rm Ker}~D$. Then the following hold:
\begin{enumerate}
\item [\rm (i)] $B$ is a faithfully flat $A$-module;
\item [\rm (ii)] the plinth ideal $\mathrm{pl}(D)$ is principal.
\end{enumerate}
\end{thm} 
\smallskip
\noindent
\begin{defn}
{\em 
Let $B=k^{[3]}$, $D \in {\rm LND}(B)$. A local slice $r$ is called a {\it minimal local slice} of $D$ if $D(r)$ is a generator of the plinth ideal $A \cap DB$.
}
\end{defn}

\smallskip
\noindent
The following result due to Freudenburg characterizes rank $1$ LNDs of $k^{[n]}$, see \cite{F95}.
\begin{prop}\label{fr95c}
Up to a change of coordinates, every locally nilpotent derivation $D$ of $rank~1$ on $k[x_1, \dots , x_n]$ is of the form $D=h\dfrac{\partial}{\partial x_n}$ where $h \in k[x_1, \dots ,x_{n-1}]$.
\end{prop}

\smallskip
\noindent
We next state an important result on the intersection of kernels of LNDs of $k^{[3]}$ \cite[Theorem~5.13]{F}.
\begin{thm}\label{3varint}
Let $D_1,D_2 \in {\rm LND}(k^{[3]})$. \mbox{Then exactly one of the following statements holds:}
\begin{enumerate}
\item [\rm (i)] ${\rm Ker}~D_1 \cap {\rm Ker}~D_2=k$;
\item [\rm (ii)] there exist $f,g,h \in B$ such that ${\rm Ker}~D_1=k[f,g]$, ${\rm Ker}~D_2=k[f,h]$ and ${\rm Ker}~D_1 \cap {\rm Ker}~D_2=k[f]$;
\item [\rm (iii)] ${\rm Ker}~D_1 ={\rm Ker}~D_2$.
\end{enumerate}
\end{thm}

\smallskip
\noindent
We are particularly interested in the case when two LNDs commute. The next result due to Daigle and Kaliman \cite[Theorem 3]{F} describes the intersection of kernels of two commuting non-equivalent LNDs of $k^{[3]}$. It was earlier observed by Maubach \cite[Theorem 3.6]{Ma} for $k=\bC$.
\begin{prop}\label{frcom}
If $D,E \in {\rm LND}(k^{[3]})$ are non-zero, commuting and have distinct kernels, then there exists a variable $f \in k^{[3]}$ such that ${\rm Ker}~D \cap {\rm Ker}~E=k[f]$.
\end{prop}

\smallskip
\noindent
The following result involving local slices of LNDs on $k^{[3]}$ is due to Daigle \cite[Lemma 2.5]{D10}.
\begin{lem}\label{dtriang}
Let $B=k^{[3]}$, $D (\neq 0) \in {\rm LND}(B)$, $A:={\rm Ker}~D$ and $f \in A$. Then the following hold:
\begin{enumerate}
\item [\rm (i)] if $f$ is a variable of $B$, then it is a variable of $A$;
\item [\rm (ii)] if $f$ is a variable of $A$ and there exists $g \in B$ such that $D(g) \in k[f] \setminus \{0\}$, then $f$ is a variable of $B$.
\end{enumerate}
\end{lem}

\smallskip
\noindent
The next result due to Freudenburg \cite[Proposition 1]{F95} is called the {\it Restriction /extension} principle.
\begin{prop}\label{frext}
Let $D \in {\rm LND}(k[x_1, \dots ,x_n])$. If $D$ restricts to a derivation on $k[x_1, \dots , x_{n-1}]$. Then $Dx_n \in k[x_1, \dots , x_{n-1}]$.
\end{prop}

\smallskip
\indent
Now let us consider the polynomial ring in $n$ variables $R^{[n]}$ over a ring $R$. By an $R$-derivation, we mean a derivation vanishing on $R$. If we are given $n-1$ polynomials $f_1,\ldots, f_{n-1}$, we can consider the Jacobian derivation $\delta={\rm Jac}(f_1,\ldots, f_{n-1},\cdot)$ such that
$$
\delta(h)={\rm Jac}(f_1,\ldots, f_{n-1},h)=\det J(f_1,\ldots, f_{n-1},h),
$$
where $J$ is the Jacobian matrix. The Jacobian derivation is always an $R$-derivation, but it is not always locally nilpotent.

\smallskip
\indent
Let us recall the concept of local slice construction introduced in~\cite{Fr97}. Put $B=k^{[3]}$. Suppose $D\in\mathrm{LND}(B)$ is irredducible and there exist such $f,g,r\in B$, $r\notin gB$ and $P\in k^{[1]}$ that $\mathrm{Ker}\,D=k[f,g]$ and $D(r)=gP(f)\neq 0$. Then there exists $\phi\in k[f]^{[1]}$ such that $\phi(r)\in gB$. We assume $\phi$ to be of minimal $r$ degree with such property. Also assume $\phi(r)$ to be irreducible in $k[f,r]$. Denote $h=\phi(f)g^{-1}\in B$. 
\begin{thm}\cite[Theorem~2]{Fr97}\label{locslcon}

(a) $\Delta={\rm Jac}(f,h,\cdot)$ is an LND of $B$;

(b) $\Delta(r)=-h P(f)$;

(c) If $\Delta$ is irreducible, then $\mathrm{Ker}\,\Delta=k[f,h]$.
\end{thm}
We say that $\Delta$ is obtained from $D$ by a local slice construction from the data $(f,g,r)$.

\smallskip
\noindent
We end this section by stating a theorem of Daigle and Freudenburg which characterizes all $R$-LNDs of $R^{[2]}$, 
where~$R$ is a UFD containing $\bQ$ \cite[Theorem 2.4]{DF}.
\begin{thm}\label{df}
Let $R$ be a UFD containing $\bQ$ with field of fractions $K$ and let $B=R[x,y]=R^{[2]}$. 
For an $R$-derivation $D\neq0$ of $B$, the following are equivalent:
\begin{enumerate}
 \item[\rm (i)] $D$ is locally nilpotent;
 \item[\rm (ii)] $D=\alpha {\rm Jac}(F,\cdot)$, for some $F \in B$ which is a variable of $K[x,y]$
  satisfying\\
  $gcd_B(\frac{\partial F}{\partial x},\frac{\partial F}{\partial y})=1$, and for some 
 $\alpha \in R[F] \setminus \{0\}$.
 \end{enumerate}
Moreover, if the above conditions are satisfied, then $Ker$ $D=R[F]=R^{[1]}$.
\end{thm}

\section{Commuting LND Construction}\label{cln}
Let $B$ be an affine $k$-domain and $D_1,D_2 \in {\rm LND}(B)$. One may ask whether the derivation $D_1+D_2$ is also an LND. In general, it is not true but the answer is affirmative when $D_1$ and $D_2$ commute. However, this property does not hold for replicas of LNDs, i.e.,  if we have two commuting LNDs, the sum of their replicas may not be locally nilpotent. 

\begin{ex}
{\em 
Consider two commuting LNDs $\frac{\partial}{\partial x}$ and $\frac{\partial}{\partial y}$ of $k[x,y]$. The sum of their replicas $y\frac{\partial}{\partial x}+x\frac{\partial}{\partial y}$ is not locally nilpotent.
}
\end{ex}

\smallskip
\noindent
The following lemma shows that if we take replicas of a special form, then their sum is an~LND.
\begin{lem}\label{lndcommute}
Let $B$ be an affine $k$-domain and $\delta_1,\ldots, \delta_m \in {\rm LND}(B)$ be non-zero and pairwise commuting, i.e., $\delta_i\delta_j=\delta_j\delta_i$ for all $i,j \in \{1, \dots ,m\}$. Suppose 
$$f_k\in \bigcap_{i=k}^m {\rm Ker}~ \delta_i, ~~~ 1\leqslant k\leqslant m.$$ 
Then $\Delta=\sum\limits_{i=1}^n f_i\delta_i$ is an LND of $B$.
\end{lem}
\begin{proof}
It is easy to see that $\Delta$ is a derivation of $B$. So it is enough to check that it is locally nilpotent. Let us fix $g\in B$. We define inductively positive integers $l_j$ from $j=1$ to $j=m$ by $l_j={\rm deg}_{\delta_j}\left(g\prod_{i=j+1}^mf_i^{l_i} \right)+1$. Set $e=l_1+\ldots+l_m-m+1$. Then, for all $a_1,a_2,\dots  ,a_e \in \{1, \dots ,m\}$,
$$
f_{a_1}\delta_{a_1}\circ\ldots\circ f_{a_e}\delta_{a_e}(g)=0.
$$
Indeed, if $m$ occurs among $a_1,\ldots, a_e$ at least $l_m$ times, then since $\delta_m$ commutes with $f_i$ and~$\delta_i$ we have for some~$\psi$:
$$
f_{a_1}\delta_{a_1}\circ\ldots\circ f_{a_e}\delta_{a_e}(g)=\psi\circ \delta_m^{l_m}(g)=0.
$$
Note, that $\delta_{m-1}$ commutes with all $f_i$ and $\delta_i$, except $f_m$. If $m$ occurs among $a_1,\ldots, a_e$ less than $l_m$ times and $m-1$ occurs among $a_1,\ldots, a_e$ at least $l_{m-1}$ times, then we have for some~$\psi$:
$$
f_{a_1}\delta_{a_1}\circ\ldots\circ f_{a_e}\delta_{a_e}(g)=\psi\circ\delta_{m-1}^{b_1}f_m^{c_1}\circ\ldots\circ\delta_{m-1}^{b_t}f_m^{c_t}\circ\delta_{m-1}^{b_{t+1}}(g),
$$
where $\sum \limits_{i=1}^{t+1} b_i\geqslant l_{m-1}$, $\sum \limits_{i=1}^t c_i< l_m$. We have 
$$\deg_{\delta_{m-1}}\left(\delta_{m-1}^{b_1}f_m^{c_1}\circ\ldots\circ\delta_{m-1}^{b_t}f_m^{c_t}\circ\delta_{m-1}^{b_{t+1}}(g)\right)\leqslant \deg_{\delta_{m-1}}\left(g f_m^{l_m} \right)-l_{m-1}<0.$$
This means that $\delta_{m-1}^{b_1}f_m^{c_1}\circ\ldots\circ\delta_{m-1}^{b_t}f_m^{c_t}\circ\delta_{m-1}^{b_{t+1}}(g)=0$.

Similarly we can consider cases when $i$ occurs among $a_1,\ldots, a_e$ less than $l_i$ times for all $i>j$ and $j$ occurs among $a_1,\ldots, a_s$ at least $l_i$ times. In all these cases we have 
$$
f_{a_1}\delta_{a_1}\circ\ldots\circ f_{a_e}\delta_{a_e}(g)=0.
$$
But since $e>\sum \limits_{i=1}^m(l_i-1)$, we are in one of these cases.

If we expand $\Delta^e(g)$, we obtain a sum of expressions $f_{a_1}\delta_{a_1}\circ\ldots\circ f_{a_e}\delta_{a_e}(g)$. Therefore, $\Delta^e(g)=0$.
\end{proof}

\begin{defn}
{\em 
We say that $\Delta$ in Lemma \ref{lndcommute}, is obtained by a {\it Commuting LND Construction} (denoted by $C$-{\it construction}) from $\delta_1,\ldots, \delta_m$.
}
\end{defn}
\begin{lem}
The LNDs $\Delta$ and $\delta_m$ commute.
\end{lem}
\begin{proof}
We have
$$
\Delta\circ\delta_m=\left(\sum\limits_{i=1}^n f_i\delta_i\right)\circ\delta_m=\sum\limits_{i=1}^n f_i\delta_i\circ\delta_m=\sum\limits_{i=1}^n f_i\delta_m\circ\delta_i=\sum\limits_{i=1}^n \delta_m\circ f_i\delta_i=\delta_m\circ\left(\sum\limits_{i=1}^n f_i\delta_i\right)=\delta_m\circ\Delta.
$$
\end{proof}

Suppose $m=2$. Starting with two commuting LNDs $\delta_1$ and $\delta_2$, by C-construction,  we obtain two commuting LNDs $\delta_2$ and $\delta_3=\Delta$. Then we can start with these two commuting LNDs and again obtain by C-construction two commuting LNDs $\delta_3$ and $\delta_4$. Iterating this procedure, we obtain a chain $\delta_1,\delta_2,\delta_3,\ldots$ , where for each $i \geqslant 2$, the derivation $\delta_i$ commutes with both $\delta_{i-1}$ and $\delta_{i+1}$. We call this chain a {\it C-chain}. A C-chain is called {\it irreducible} if all the LNDs $\delta_i$ are irreducible.

\subsection{Two variables}\label{sectwovar}

Let $B=k[x,y]$ and $\delta_1, \delta_2$ be two commuting non-equivalent LNDs. Then C-construction gives $\delta_3=\lambda \delta_1+\sigma\delta_2$, where $\sigma\in {\rm Ker}~\delta_2$, $\lambda\in k$ and $\delta_3$ commutes with $\delta_2$. It turns out that in case when $\delta_2$ is irreducible, each LND commuting with $\delta_2$ can be obtained by C-construction from $\delta_1$ and $\delta_2$.

\begin{lem}\label{dimtwo}
Suppose $D$ and $E$ are commuting non-equivalent LNDs of $B$. Assume $D$ is irreducible. Let $E'$ be an LND of $B$ commuting with $D$. Then $E'=\lambda E+\sigma D$, where $\sigma\in {\rm Ker}~ D$ and $\lambda \in k^*$.
\end{lem}
\begin{proof}
By Theorem \ref{R}, there exist $f,g\in k[x,y]$ such that $k[x,y]=k[f,g]$ and $D=\frac{\partial}{\partial f}$. Let us consider restrictions $\varepsilon$ and $\varepsilon'$ of $E$ and $E'$ respectively to ${\rm Ker}~ D=k[g]$. Since $DE=ED$ and $DE'=E'D$, both $\varepsilon$, $\varepsilon' \in {\rm LND}(k[g])$.  So we have $\varepsilon=\mu\frac{\partial}{\partial g}$, $\varepsilon'=\mu'\frac{\partial}{\partial g}$, where $\mu, \mu' \in k$.  If $\mu=0$, then $E|_{{\rm Ker}~ D}=0$ and consequently $E$ and $D$ are equivalent, which is a contradiction. So $\mu \neq 0$. Let $\lambda=\frac{\mu'}{\mu}$ and $F=E'-\lambda E$. Then $FD=DF$ and $F(g)=0$. Since $F(g)=0$, we have $F=\sigma\frac{\partial}{\partial f}$ for some $\sigma\in B$. We have 
$\frac{\partial\sigma}{\partial f}=D(\sigma)=DF(f)=FD(f)=F(1)=0$. So, $E'=\sigma D+\lambda E$, where $\sigma=\sigma(g)\in {\rm Ker}~ D$ and $\lambda \in k^*$.
\end{proof}

\begin{cor}\label{corrin}
Let $R$ be an affine $k$-domain with field of fractions $K$. Suppose $D$ and $E$ are commuting non-equivalent $R$-LNDs of $R[x,y]$. Assume $D$ is irreducible. Let $E'$ be an $R$-LND of $R[x,y]$ commuting with $D$. Then there exist non-zero elements $s,r \in R$ and $q \in {\rm Ker}~D$ such that $sE'=r E+q D.$
\end{cor}
\begin{proof}
The LNDs  $D$, $E$, and $E'$ induce LNDs of $K[x,y]$. By Lemma~\ref{dimtwo}, $E'=\lambda E+\sigma D$, where $\lambda\in K^*$, $\sigma\in K[x,y]$ and $\sigma \in {\rm Ker}~D$. There exist non-zero elements $s,r \in R$ such that $r=s\lambda\in R$ and $q=s\sigma\in R[x,y]$. Then $sE'=r E+q D$.
\end{proof}

\smallskip
\noindent 
Let $\delta \in {\rm LND}(k[x,y])$ be irreducible. By Theorem \ref{R}, there exist $f,g\in k[x,y]$ such that $k[x,y]=k[f,g]$ and $\delta=\frac{\partial}{\partial g}$. Let $\alpha$ be the automorphism of $k[x,y]$ given by $(x,y)\mapsto (f,g)$. Let us denote by $\Psi(\delta)$ the right coset $T_2(k)\alpha$. The {\it association class} of an LND $\delta$, denoted by $[\delta]$, is defined to be $\{ \zeta \in {\rm LND}(B) ~\vert~ \zeta=\lambda \delta, \lambda \in k^* \}$.

\begin{lem}\label{bij}
The mapping $\Psi$ gives a bijection between association classes of irreducible LNDs of $k[x,y]$ and right cosets of $T_2(k)$.
\end{lem}
\begin{proof}
Let $\delta \in {\rm LND}(k[x,y])$ be irreducible. Suppose $\alpha : (x,y)\mapsto (f,g)$ and $\alpha_1\colon(x,y)\mapsto (f_1,g_1)$ be two distinct automorphisms such that $\dfrac{\partial }{\partial g}=\delta=\dfrac{\partial }{\partial g_1}.$
 Then ${\rm Ker}~ \delta=k[f]=k[f_1]$. Therefore, $f_1=af+b$, where $a \in k^*$ and $b \in k$. Then, we have $g_1=cg+P(f)$ where $c \in k^*$ and $P(f) \in k[f]$. So, right cosets $\mathrm{T}_2(k)\alpha$ and $\mathrm{T}_2(k)\alpha_1$ coincide.  Now let us consider $\lambda\delta$, $\lambda\in k^*$. It corresponds to the automorphism $\alpha' : (x,y)\mapsto (f,\frac{g}{\lambda})$. Then $\Psi(\delta)=\mathrm{T}_2(k)\alpha=\mathrm{T}_2(k)\alpha'=\Psi(\lambda\delta)$. Thus $\Psi$ induces a well-defined map from the set of association classes of $\delta$ to the set of right cosets of ${\rm T}_2(k)$. For convenience, we also call that map $\Psi$. 

\smallskip
\noindent 
Suppose $\mathrm{T}_2(k)\alpha=\mathrm{T}_2(k)\alpha_1$. Then $f_1=af+b$ and $g_1=cg+P(f)$, where $a,c \in k^*$, $b \in k$ and $P(f) \in k[f]$. Therefore, $\frac{\partial}{\partial g_1}=\frac{1}{c}\frac{\partial}{\partial g}$. This provides injectivity of $\Psi$. For each automorphism $\alpha\colon(x,y)\mapsto (f,g)$ we have $\mathrm{T}_2(k)\alpha=\Psi\left(\frac{\partial}{\partial g}\right)$. Therefore, $\Psi$ is surjective. 
\end{proof}

\begin{lem}\label{hl}
Let $\beta \in {\rm Af}_2(k)$ and $\alpha \in {\rm Aut}(k^{[2]})$. Let $D,E \in {\rm LND}(k^{[2]})$ be such that for the association classes $[D]$ and $[E]$, we have $\Psi([D])=\mathrm{T}_2(k)\alpha$ and $\Psi([E])=\mathrm{T}_2(k)\beta\circ\alpha$. Then $D$ and $E$ commute.
\end{lem}
\begin{proof}
Denote $\alpha=(f,g)$. We can assume that $D(f)=0$, $D(g)=1$. Then
$$\beta\circ\alpha\colon
\begin{pmatrix}
x\\
y
\end{pmatrix}\mapsto
\begin{pmatrix} a_{11} & a_{12} \\
a_{21} & a_{22} \end{pmatrix} \begin{pmatrix}
f\\
g
\end{pmatrix}+
\begin{pmatrix}
b_1\\
b_2
\end{pmatrix},
$$ where $\begin{pmatrix} a_{11} & a_{12} \\
a_{21} & a_{22} \end{pmatrix} \in {\rm GL}_2(k)$ and $b_1,b_2 \in k$.
We can further assume that 
$$E(a_{11}f+a_{12}g+b_1)=a_{11}E(f)+a_{12}E(g)=0\text{ and }E(a_{21}f+a_{22}g+b_2)=a_{21}E(f)+a_{22}E(g)=1.$$ 
Therefore, $E(f), E(g) \in k$. So, we have
$$
ED(f)=E(0)=0=DE(f), \qquad ED(g)=E(1)=0=DE(g).
$$ 
Thus, $ED=DE$.
\end{proof}

\begin{prop}\label{CLNDC2var}
Each irreducible $\delta \in {\rm LND}(k^{[2]})$ can be constructed from an irreducible C-chain starting from $\delta_1=\frac{\partial}{\partial x}$ and $\delta_2=\frac{\partial}{\partial y}$. \end{prop}
\begin{proof}
Let $\alpha\in\Psi(D)$. By Theorem \ref{jvk}, we can decompose $\alpha=\xi_m\circ\ldots\circ\xi_1$, where each $\xi_j \in {\rm Af}_2(k) \cup {\rm T}_2(k)$. Consider sequence of right cosets 
$$\mathrm{T}_2(k), \mathrm{T}_2(k)\xi_1, \mathrm{T}_2(k)\xi_2\circ \xi_1, \ldots, \mathrm{T}_2(k)\xi_m\circ\ldots\circ\xi_1.$$
If two consecutive cosets in the sequence are equal, we remove one of them. Thus we obtain sequence of right cosets $\mathrm{T}_2(k)\gamma_1,\ldots, \mathrm{T}_2(k)\gamma_l$, where $\gamma_i \in {\rm Aut}(k^{[2]})$, $1 \leqslant i \leqslant l$.  Let us take $\delta_i \in {\rm LND}(k^{[2]})$ such that $\Psi(\delta_i)=\mathrm{T}_2(k)\gamma_i$. For any $i$, if $\mathrm{T}_2(k)\gamma_i =\mathrm{T}_2(k)\xi_s\circ\ldots\circ\xi_1$ and $\mathrm{T}_2(k)\gamma_{i+1} =\mathrm{T}_2(k)\xi_{s+1}\circ\xi_s\circ\ldots\circ\xi_1$, then  $\xi_{s+1} \in {\rm Af}_2(k) \setminus {\rm T}_2(k)$. By Lemma \ref{hl}, $\delta_i$ commutes with $\delta_{i-1}$ and $\delta_{i+1}$. Hence, by Lemma~\ref{dimtwo}, $\delta_{i+1}=\sigma \delta_i+\lambda \delta_{i-1}$, where $\sigma\in {\rm Ker}~\delta_i$, $\lambda\in k^*$. Since  $\alpha=\xi_m\circ\ldots\circ\xi_1$, we can put $\delta_{l}=\delta$. We start with $\gamma_1=\mathrm{id}$, $\delta_1=\frac{\partial}{\partial y}$ and define $\delta_0:=\frac{\partial}{\partial x}$. By construction, $\delta_2$ commutes with $\delta_1$. By Lemma~\ref{dimtwo}, $\delta_2$ is obtained by C-construction from $\delta_0$ and $\delta_1$. Moreover, it is irreducible. So, $\delta_0=\frac{\partial}{\partial x},\delta_1=\frac{\partial}{\partial y},\ldots, \delta_l=\delta$ is an irreducible C-chain.

\end{proof}

\begin{defn}{\em 
A C-chain, starting from $\frac{\partial}{\partial x}$ and $\frac{\partial}{\partial y}$ will be called a {\it full C-chain}.
} \end{defn}

\smallskip
\noindent
The next result is a partial converse of Lemma \ref{hl}
\begin{lem}\label{com2}
Suppose $\Delta$ and $D$ are commuting irreducible LNDs of $k^{[2]}$. Then there exist such $\alpha\in\Psi(D)$ and $\theta\in \mathrm{Af}_2(k)$ that $\theta\circ\alpha\in \Psi(\Delta)$.
\end{lem}
\begin{proof}
It is enough to consider the case when $\Delta$ and $D$ are non-equivalent. Indeed, if $\Delta=tD$ for some $t \in k^*$, then $\Psi(\Delta)=\Psi(D)$. Now suppose we have $D=\frac{\partial}{\partial g}$, where $\tau\colon (x,y)\mapsto(f,g)$ is an automorphism from $\Psi(D)$. Denote $E=\frac{\partial}{\partial f}$. Then $DE=ED$. By Lemma \ref{dimtwo}, $\Delta=\sigma D+\lambda E$, where $\sigma\in{\rm Ker}~ D=k[f]$ and $\lambda \in k^*$. Let $h\in k[f]$ be such that $\frac{d h}{df}=\sigma$. We have $\Delta(f)=\lambda$, $\Delta(g)=\sigma$. Therefore $\Delta(\lambda g-h(f))=0$. Consider the automorphism
$$
\gamma\colon
\begin{pmatrix}
x\\
y
\end{pmatrix}\mapsto
\begin{pmatrix}
x\\
\lambda y-h(x)
\end{pmatrix}
$$
Clearly, $\gamma\in\mathrm{T}_2(k).$
Then 
$$
\alpha=\gamma\circ\tau\colon
\begin{pmatrix}
x\\
y
\end{pmatrix}\mapsto
\begin{pmatrix}
f\\
\lambda g-h(f)
\end{pmatrix}.
$$
Let us put 
$$
\theta\colon 
\begin{pmatrix}
x\\
y
\end{pmatrix}\mapsto
\begin{pmatrix}
y\\
x
\end{pmatrix},\qquad \theta\in \mathrm{Af}_2(k).
$$
We have $\alpha\in \mathrm{T}_2(k)\tau=\Psi(D)$. Choose $\delta \in {\rm LND}(k[x,y])$ such that $\Psi([\delta])=\mathrm{T}_2(k)\theta\circ\alpha$. Then $\delta(f)=\mu$, for some $\mu \in k^*$ and $\delta(\lambda  g-h(f))=0$. Hence, $\delta(g)=\frac{\mu\sigma(f)}{\lambda}$. So, $\Delta=\frac{\lambda}{\mu}\delta$, i.e. $\theta\circ\alpha\in\Psi(\Delta)$.
\end{proof}

\smallskip
\noindent
We immediately have the following corollary.
\begin{cor}\label{c1}
Let $\delta_1, \delta_2,\ldots, \delta_l$ be a full irreducible C-chain in $k^{[2]}$. Then there exists a chain of automorphisms $\alpha_2=\mathrm{id},\beta_2,\ldots, \alpha_l$ such that $\beta_i=\tau_i\circ\alpha_i$, $\alpha_{i+1}=\gamma_i\circ\beta_i$, where $\gamma_i\in\mathrm{Af}_2(k)$, $\tau_i\in\mathrm{T}_2(k)$, and  $\alpha_{i},\beta_{i}\in \Psi(\delta_i)$, for each $i \in \{1, \dots ,l\}$.
\end{cor}
\begin{proof}
We have $\alpha_2=\mathrm{id}\in\Psi(\frac{\partial}{\partial y})=\Psi(\delta_1)$. 
For each $i$, by Lemma~\ref{com2}, there exist $\beta_i\in \Psi(\delta_i)$ and $\gamma_i\in \mathrm{Af}_2(k)$ such that $\alpha_{i+1}:=\gamma_i\circ\beta_i\in \Psi(\delta_{i+1})$. Since $\alpha_{i}, \beta_{i}\in \Psi(\delta_i)=\mathrm{T}_2(k)\alpha_i$, we have $\beta_i=\tau_i\circ\alpha_i$, where $\tau_i\in \mathrm{T}_2(k)$.
\end{proof}

\smallskip
\noindent
Next we define the level of an LND.
\begin{defn}
{\em Let $\delta$ be an irreducible LND of $k[x,y]$. Assume $\delta_1, \delta_2,\ldots, \delta_l=\delta$ is a full irreducible C-chain with minimal possible $l$. Let us define the {\it level} of $\delta$ to be this integer $l$ and denote it by $\mathrm{lev}(\delta)_{(x,y)}$. 
}
\end{defn}

\begin{rem}
{\em The definition of level depends on the coordinate system chosen for $k^{[2]}$. Indeed, for an irreducible $\delta \in {\rm LND}(k^{[2]})$, there exist $f,g \in k^{[2]}$ such that $k[x,y]=k[f,g]$ and $\delta=\dfrac{\partial }{\partial f}$. Then $\mathrm{lev}(\delta)_{(f,g)}=1$. So for any irreducible LND of $k^{[2]}$, there exists a coordinate system of $k^{[2]}$ with respect to which the level is one.
}
\end{rem}

\begin{rem}\label{rr11}
{\em Let $\varphi\in\Psi(\delta)$. Assume $\delta \notin \Big [\dfrac{\partial}{\partial x}\Big ] \cup \Big[\dfrac{\partial}{\partial y}\Big ]$. Then Corollary~\ref{c1} implies that 
$\mathrm{lev}(\delta)_{(x,y)}=m+2$, where $m$ is the number of affine automorphisms in the shortest decomposition 
$$\varphi=\tau_m\circ\gamma_m\ldots\circ\tau_1\circ\gamma_1,$$ 
where each $\gamma_i \in {\rm Af}_2(k)$ and $\tau_i \in {\rm T}_2(k)$ ($1 \leqslant i \leqslant m$). By Theorem~\ref{jvk}, the condition that this decomposition is shortest is equivalent to saying that $\gamma_i\notin \mathrm{T}_2(k)$, $1\leqslant i\leqslant m$ and $\tau_j\notin \mathrm{Af}_2(k)$, $1\leqslant j\leqslant m-1$. Note that, each $\gamma_i$ can be decomposed as $\gamma_i=\tau'\circ\theta\circ\tau''$, where $\tau', \tau''\in \mathrm{T}_2(k)$ and $\theta$ is the swapping
$$
\theta\colon
\begin{pmatrix}
x\\
y
\end{pmatrix}
\mapsto
\begin{pmatrix}
y\\
x
\end{pmatrix}.
$$  Therefore, we can say that $\mathrm{lev}(\delta)_{(x,y)}=m+2$, where $m$ is the number of swappings in the shortest decomposition 
$\varphi=\tau_m\circ\theta\ldots\circ\tau_1\circ\theta\circ\tau_0$. 
}
\end{rem}

\smallskip
\noindent
The next proposition shows that given an irreducible $\delta \in {\rm LND}(k^{[2]})$ with level $l$ with respect to a fixed coordinate system, the LNDs in a full irreducible minimal C-chain are unique up to multiplication by scalars.
\begin{prop}
Let $\delta$ be an irreducible LND of $k[x,y]$ with $\mathrm{lev}(\delta)_{(x,y)}=l$. Assume 
$$\delta_1, \delta_2,\ldots, \delta_l=\delta\  \text{ and }\ \rho_1, \rho_2,\ldots, \rho_l=\delta$$ 
are two full irreducible minimal C-chains with length $l$ . Then for each $i \in \{1, \dots ,l\}$, we have $\rho_i=\lambda_i\delta_i$ for some non-zero $\lambda_i\in k$.
\end{prop}
\begin{proof}
From two full irreducible minimal C-chains, by Corollary \ref{c1}, we obtain two sequences of automorphisms: $\alpha_{1}=\mathrm{id},\beta_1,\alpha_2,\beta_2,\ldots, \alpha_l$ and $\alpha'_{1}=\mathrm{id},\beta'_1,\alpha'_2,\beta'_2,\ldots, \alpha'_l$. By construction, we have $\alpha_l,\alpha'_l\in \Psi(\delta)$. Hence, $\alpha_l'\in \mathrm{T}_2(k)\alpha_l$. So there exists $\omega\in\mathrm{T}_2(k)$ such that $\alpha'_l=\omega\circ \alpha_l$. Note that, $\beta_i=\tau_i\circ\alpha_i$ and $\alpha_{i+1}=\gamma_i\circ\beta_i$, where $\gamma_i\in\mathrm{Af}_2(k)$, $\tau_i\in\mathrm{T}_2(k)$. Similarly, $\beta'_i=\tau'_i\circ\alpha'_i$, $\alpha'_{i+1}=\gamma'_i\circ\beta'_i$, where $\gamma'_i\in\mathrm{Af}_2(k)$, $\tau'_i\in\mathrm{T}_2(k)$. So, we obtain
$$
\gamma'_{l-1}\circ\tau'_{l-1}\circ\ldots\circ\gamma'_1\circ\tau'_1=\alpha'_l=\omega\circ \alpha_l=\omega \circ \gamma_{l-1}\circ\tau_{l-1}\circ\ldots\circ\gamma_1\circ\tau_1.
$$

Since $l$ is minimal, we have $\gamma_{i}, \gamma'_{i}\notin \mathrm{T}_2(k)$ and $\tau_j,\tau'_j\notin\mathrm{Af}_2(k)$. By Theorem~\ref{jvk}, we then have $\alpha'_j=\eta_j\circ\alpha_j$ for each $j \in \{1, \dots ,l\}$, where $\eta_j\in\mathrm{T}_2(k)\cap \mathrm{Af}_2(k)$. Therefore, $\Psi^{-1}(\alpha_j)=\Psi^{-1}(\alpha'_j)$. Since $\alpha_j \in \Psi(\delta_j)$, $\alpha'_j \in \Psi(\rho_j)$ and ${\rm T}_2(k)\alpha_j={\rm T}_2(k)\alpha'_j$, by Lemma \ref{bij}, we are done. 
\end{proof}

\subsection{LNDs of rank at most $2$}
This subsection shows that one can apply the C-construction to construct LNDs of $k^{[n]}$ of rank at most $2$. Let $B=k[x_1,\ldots,x_n]$ and $D \in {\rm LND}(B)$ such that ${\rm rk}~D \leqslant 2$. We can choose a system of variables such that $k[x_3,\ldots,x_n]\subseteq \mathrm{Ker}\,D$. Then we can obtain $D$ from partial derivations by a construction similar to the C-construction.

\smallskip
\noindent
\begin{lem}\label{CLNDC3var1}
Let $R$ be an affine $k$-domain with field of fractions $K$ and $D$ an irreducible $R$-LND of $R[x_1,x_2](=R^{[2]})$. Then there exists a chain of irreducible LNDs of $R[x_1,x_2]$  
$$ \partial_1:=\dfrac{\partial }{\partial x_1}, \partial_2 :=\dfrac{\partial }{\partial x_2}, \partial _3, \dots , \partial_l=D,$$ 
such that for all $i$, $1 \leqslant i \leqslant l$, $\partial_i(R)=0$ and 
$$h_i\partial_i=\sigma \partial_{i-1}+f_i\partial_{i-2}$$ 
for some $\sigma\in {\rm Ker}~\partial_{i-1}$, $f_i,h_i\in R\setminus \{0\}$.
\end{lem}

\begin{proof}
Since $D(R)=0$, it induces an LND $\delta$ of $K[x_1,x_2]$. Since $D$ is irreducible in $R[x_1,x_2]$, so is $\delta$ in $K[x_1,x_2]$. By Proposition \ref{CLNDC2var}, $\delta$ can be constructed by a full irreducible C-chain $\{\delta_i\}_{i=1}^l$ in $K[x_1,x_2]$ starting from $\delta_1=\dfrac{\partial }{\partial x_1}$ and $\delta_2=\dfrac{\partial }{\partial x_2}$. For each $i$, there exists $g_i \in K^*$ such that $\partial_i=g_i\delta_i$ is an irreducible LND of $R[x_1,x_2]$. Note that $\partial_{i-2}$ and $\partial_i$ commutes with~$\partial_{i-1}$. By Corollary~\ref{corrin}, we have $h_i\partial_i=\sigma \partial_{i-1}+f_i\partial_{i-2}$ for some $\sigma\in \mathrm{Ker}\,\partial_{i-1}$, $f_i,h_i\in R\setminus \{0\}$.
\end{proof}

\smallskip
\noindent
We immediately have the following corollary.
\begin{cor}\label{CLNDC3var} Let $D \in {\rm LND}(k[x_1,\ldots,x_n])$ be irreducible and $D(x_3)=\ldots=D(x_n)=0$. Then there exists a chain of irreducible LNDs of 
$k[x_1,\ldots,x_n]$ 
$$ \partial_1:=\dfrac{\partial }{\partial x_1}, \partial_2 :=\dfrac{\partial }{\partial x_2}, \partial _3, \dots , \partial_l=D,$$ 
such that for all $i$, $1 \leqslant i \leqslant l$, we have $\partial_i(x_3)=\ldots=\partial_i(x_n)=0$ and 
$$h_i\partial_i=\sigma \partial_{i-1}+f_i\partial_{i-2}$$ 
for some $\sigma\in \mathrm{Ker}\,\partial_{i-1}$, $f_i,h_i\in k[x_3,\ldots,x_n]\setminus \{0\}$.
\end{cor}

\begin{proof} The proof follows from Lemma \ref{CLNDC3var1} by setting $R=k[x_3, \dots , x_n]$.
\end{proof}

We call this construction a {\it Modified C-construction} (denoted by {\it MC-construction}) and the chain $\partial_1,\ldots, \partial_l$ a {\it full irreducible MC-chain}.
Like in the case of $k^{[2]}$, one can define the level of an irreducible LND $D$ of rank at most $2$, as the length of the shortest full irreducible MC-chain ending at $D$. 
As noted earlier, this level depends on the chosen coordinate system. We denote it by $\mathrm{lev}_{\mathbf{x}}(D)$, where $\mathbf{x}=\{x_1,x_2,\ldots,x_n\}$. 
This motivates us to make the following definition. Suppose $D$ is an LND of $k^{[n]}$ with rank at most $2$. Let $\Gamma_D$ denote the collection of all coordinate systems $\mathbf{y}=\{y_1,\ldots,y_n\}$  of $k^{[n]}$ such that $D(y_i)=0$ for $i\geqslant 3$. We would like  to define the universal level as minimal possible level among all such coordinate systems~$\mathbf{y}$.
 
\begin{defn}
{\em For an irreducible $D \in {\rm LND}(k^{[n]})$ of rank at most $2$, we define the {\it universal level} of $D$, denoted by $\mathrm{ulev}(D)$ as 
$$\mathrm{ulev}(D)=\min_{\mathbf{y} \in \Gamma_D}~\left\{ \mathrm{lev}_{\mathbf{y}}(D)~\vert ~ \mathbf{y} \in \Gamma_D \right\}.$$
}
\end{defn}

\smallskip
\noindent
It easy to see that $\mathrm{ulev}(D)$ does not attain the value $2$. Using universal level, one can characterize rank $1$ LNDs and triangularizable LNDs of $k^{[n]}$. 
\begin{thm}\label{crtr} Let $D \in {\rm LND}(k^{[n]})$ be irreducible of $rank$ at most $2$. Then the following statements hold:
\begin{enumerate}
\item [\rm (i)] $\mathrm{ulev}(D)=1 \Leftrightarrow rk\,D=1$; 
\item [\rm (ii)] $\mathrm{ulev}(D)= 3 \Leftrightarrow D~\text{is triangularizable LND of rank 2}.$
\end{enumerate}
\end{thm}

\begin{proof} (i) Suppose $\mathrm{ulev}(D) =1$. Then there exists a coordinate system $\{y_1,\ldots,y_n\}$ of $k^{[n]}$ such that $D=\dfrac{\partial }{\partial y_1}$. Hence, $\mathrm{rk}\,D=1$. Conversely, since $D$ is irreducible of rank $1$, there exists a coordinate system of $k^{[n]}$ in which $D$ is a partial derivation. Hence, $\mathrm{ulev}(D)=1$.

\smallskip
\noindent
(ii) The condition $\mathrm{ulev}(D)=3$ is equivalent to the fact that there exists a coordinate system $\{y_1,\ldots,y_n\}$ of~$k^{[n]}$ such that 
$$h(y_3,\ldots,y_n)D=f(y_3,\ldots,y_n)\dfrac{\partial }{\partial y_1}+\sigma(y_1,y_3,y_4,\ldots,y_n)\dfrac{\partial }{\partial y_2},$$
for some $f \in k[y_3, \ldots , y_n]$ and $\sigma \in {\rm Ker}~\dfrac{\partial}{\partial x_2}=k[y_1, y_3, \ldots ,y_n]$. 
Swapping $y_1$ and $y_2$ we obtain
$$h(y_3,\ldots,y_n)D=f(y_3,\ldots,y_n)\dfrac{\partial }{\partial y_2}+\sigma(y_2,y_3,y_4,\ldots,y_n)\dfrac{\partial }{\partial y_1}.$$ 
So, $D(y_1)=  \dfrac{\sigma(y_2,\ldots,y_n)}{h(y_3,\ldots,y_n)}$ and $D(y_2)= \dfrac{f(y_3,\ldots,y_n)}{h(y_3,\ldots,y_n)}$. Since $D(y_1), D(y_2) \in k^{[n]}$, it follows that $D(y_1) \in k[y_2, \ldots , y_n]$ and $D(y_2) \in k[y_3, \ldots ,y_n]$. Clearly, \mbox{$D(y_3)=\ldots=D(y_n)=0$.} This implies triangularizability of $D$ and since ${\rm ulev}(D) \neq 1$, we must have ${\rm rk}~D=2$ by part (i).

\smallskip
\noindent
Conversely, if $D$ is triangularizable of rank $2$, then in some coordinate system $\{y_1,\ldots,y_n\}$, we have 
$$D(y_3)=\ldots=D(y_n)=0,\qquad D(y_2)=f(y_3,\ldots,y_n),\qquad D(y_1)=h(y_2,\ldots,y_n), $$ for suitable non-zero $f$ and $h$.  Then
$$
D=h(y_2,\ldots,y_n)\frac{\partial}{\partial y_1}+f(y_3,\ldots,y_n)\frac{\partial}{\partial y_2}.
$$
So $\frac{\partial}{\partial y_2}, \frac{\partial}{\partial y_1}, D$ is an irreducible MC-chain. Hence $\mathrm{ulev}(D)\leqslant 3$. But by (i), $\mathrm{ulev}(D)\neq 1$. Since universal level never equals 2, we have $\mathrm{ulev}(D)=3$.
\end{proof}

\begin{rem}\label{ulev}
{\em Since there are irreducible LNDs on $k^{[3]}$ with universal level $1$ or $3$, it is natural to ask whether for an irreducible $D \in {\rm LND}(k^{[3]})$, can ${\rm ulev}(D) > 3 ?$ In Section \ref{rk2sec}, we give an example of an LND on $k[x,y,z]$ with universal level $4$ (see Example \ref{ex1}). But we don't have any example of an LND for which we can prove that the universal level is $5$ or more. So we end this section with the following question.
}
\end{rem}

\begin{que}\label{qu} {\em For an irreducible $D \in {\rm LND}(k^{[n]})$, $n\geqslant 3$, can ${\rm ulev}(D) > 4$?}
\end{que}
\noindent
We have an explicit candidate of an LND $D$ of $k^{[3]}$ with $\mathrm{ulev}(D)=5$, see Remark~\ref{varema} and Question~\ref{qq}. 
If we are given an LND $D$ of $k[x_1,x_2,\ldots,x_n]$ with $D(x_3)=\ldots=D(x_n)=0$, we can consider its extension $\delta$ to $k(x_3,\ldots, x_n)[x_1,x_2]$ and  take $\alpha\in \Psi(\delta)$. This automorphism $\alpha$ can be algorithmically decomposed onto a composition of triangular and affine ones. 
This gives us an algorithm for computing $\mathrm{lev}_{(x_1,\ldots, x_n)}D$. This leads to the following question.
\begin{que}\label{ququ}
Is there an algorithm to compute the universal level of a given irreducible LND of $k^{[n]}$ of rank at most $2$, where $n \geqslant 3$? 
\end{que}

\section{Rank $\mathbf{2}$ derivations}\label{rk2sec}
In this section, we study locally nilpotent derivations of $B=k^{[3]}$. We first give a necessary and sufficient condition for a locally nilpotent derivation to have a non-equivalent commuting locally nilpotent derivation. 

\begin{prop}\label{rk2com1}
Let $B=k^{[3]}$ and $D (\neq 0) \in {\rm LND}(B)$. Then there exists $E (\neq 0) \in {\rm LND}(B)$ such that $DE=ED$ and ${\rm Ker}~D \neq {\rm Ker}~E$ if and only if there exists $g \in B$ such that $D(g)=f(x) \in k[x]$ and $x$ is a variable of ${\rm Ker}~D$.
\end{prop}

\begin{proof} Suppose $DE=ED$ and ${\rm Ker}~D \neq {\rm Ker}~E$. Then, by Proposition \ref{frcom}, there exists a variable $x$ of $B$ such that ${\rm Ker}~D \cap {\rm Ker}~E=k[x]$. By Lemma \ref{dtriang}(i), $x$ is a variable of ${\rm Ker}~D$. Since $DE=ED$ and ${\rm Ker}~D \neq {\rm Ker}~E$, $D$ restricts to $\delta (\neq 0) \in {\rm LND}({\rm Ker}~E)$. Since~$\delta$ is non-zero, it has a local slice $g$. Then $D(g)=\delta (g) \in {\rm Ker}~\delta=k[x]$.

\smallskip
\noindent
Conversely, suppose that there exists $g \in B$ such that $D(g)=f(x) \in k[x]$ and $x$ is a variable in ${\rm Ker}~D$. Then, by Lemma \ref{dtriang}(ii), $x$ is a variable in $B$ such that $D(x)=0$. Let $B=k[x,y,z]$ and consider the derivation $$E:={\rm Jac}(x,g,\cdot)=\dfrac{\partial g}{\partial y}\dfrac{\partial}{\partial z}-\dfrac{\partial g}{\partial z}\dfrac{\partial}{\partial y}.$$
Let ${\rm Ker}~D=k[x,p]$ for some $p \in B$. By Theorem \ref{df}, we have
$$D=h~{\rm Jac}(x,p,.)=h\left(\dfrac{\partial p}{\partial y}\dfrac{\partial}{\partial z}-\dfrac{\partial p}{\partial z}\dfrac{\partial}{\partial y}\right),~\text{where}~h \in k[x,p]\setminus \{0\}.$$ Since $D(g)=f(x) \in k[x]$, we have $h \mid f(x)$. So we also have $h=h(x) \in k[x]$. Note that  
$$h(x)E(p)=h(x)\left(\dfrac{\partial g}{\partial y}\dfrac{\partial p}{\partial z}-\dfrac{\partial g}{\partial z}\dfrac{\partial p}{\partial y}\right)=-D(g)=-f(x).$$ Hence $E^2(p)=0$. Since $g$ is a local slice of $D$, we have $k[x,y,z,\frac{1}{f(x)}]=k[x,p,\frac{1}{f(x)}][g]$. So any element $s \in B$ can be represented as 
$$s=\dfrac{f_s(x,p,g)}{f^{m_s}},~\text{where}~f_s \in k[x,p,g]~\text{and}~m_s \in \bN \cup \{0\}.$$
Since $E(x)=E(g)=E^2(p)=0$, we obtain $E \in {\rm LND}(B)$. Now we have 
$$ED(x)=DE(x)=0, \qquad ED(g)=E(f(x))=0=DE(g),$$
$$h(x)DE(p)=D(h(x)E(p))=D(-f(x))=0=h(x)ED(p).$$ Since $D$ and $E$ commute on $k[x,p,g]$ and $D(x)=E(x)=0$, they also commute on $B$ because of $B[\frac{1}{f(x)}]=k[x,p,g,\frac{1}{f(x)}].$
\end{proof}

\smallskip
\noindent
\begin{rem}{\em The above proof is similar to the proof given in \cite[Lemma 4.1]{KH1}. The authors prove the existence of a locally nilpotent derivation commuting with $D$, under the hypothesis that $D$ is {\it irreducible of $rank~2$}. The condition that ``{\it there exists $g \in B$ such that $D(g)=f(x) \in k[x]$ and $x$ is a variable in ${\rm Ker}\,D$}'' is automatically satisfied if $g$ is a minimal local slice of $D$ \cite[Lemma 2.5, Theorem 3.1]{KH1} under the assumption that $\rm{rk}~D \neq 1$. Our result holds true for any locally nilpotent derivation of $B$ of $\rm{rank}$ at most $2$ without any assumption on irreducibility.
}
\end{rem}

\smallskip
\noindent
We next claim an elementary lemma.
\begin{lem}\label{rk2lem}
Let $B=k^{[3]}$ and $D \in {\rm LND}(B)$ be irreducible of rank at most 2. Then there exists a variable $x$ of $B$ and $g \in B$ such that $D(g)=f(x) \in {\rm Ker}~D$.
\end{lem}

\begin{proof} Since $\mathrm{rk}\,D \leqslant 2$, there exists a variable $x$ of $B$ such that $x \in {\rm Ker}~D$. Then $D$ extends to an irreducible locally nilpotent derivation $\tilde{D}$ on $k(x)[y,z]$. By Theorem \ref{R}, $\tilde{D}$ has a slice, i.e., there exist $g \in B$ and $f \in k[x]\setminus \{0\}$ such that $\tilde{D}(\frac{g}{f})=1$. Hence $D(g)=f(x)$.
\end{proof}

\smallskip
\noindent
We immediately have the following corollary.
\begin{cor}\label{var}
Let $B=k^{[3]}$, $D \in {\rm LND}(B)$ be of rank~2 and $A:={\rm Ker}~D$. Let $x \in A$ and  $x^{\prime} \in A$ be any other variable of $B$, then there exist $\lambda \in k^*$ and $\mu \in k$ such that $x^{\prime}=\lambda x + \mu$.
\end{cor}

\begin{proof} Without loss of generality we may assume $D$ to be irreducible. By Lemma \ref{rk2lem}, there exist $g, g^{\prime} \in B$ such that $D(g)=f(x)$ and $D(g^{\prime})=h(x^{\prime})$ for some $f \in k[x]\setminus \{0\}$ and $h \in k[x^{\prime}]\setminus \{0\}$. By Theorem \ref{pl} (ii), there exists $v \in A$ such that $A \cap DB=(v)$. Then $v \mid f$ and $v \mid h$. So there exists $a(x) \in k[x]$ and $b(x^{\prime}) \in k[x^{\prime}]$ such that $v=a(x)=b(x^{\prime})$. Since $\rm{rk}~D \neq 1$, $v \notin k$. So $x$ and $x^{\prime}$ are algebraically dependent over $k$. Since $k[x]$ is algebraically closed in $B$, we have $x^{\prime} \in k[x]$. But $x^{\prime}$ is a variable of $B$ and hence has to be a linear polynomial in $x$. Thus there exist $\lambda \in k^*$ and $\mu \in k$ such that $x^{\prime}=\lambda x + \mu$.
\end{proof}

\smallskip
Let us now describe all LNDs of $B$ commuting with a fixed non-zero irreducible LND $D$.
\begin{thm}\label{rk2thm1} Let $D \in {\rm LND}(B)$ be irreducible. Then the following statements hold:
\begin{enumerate}
\item [\rm (i)] if $\mathrm{rk}\,D=3$, then all LNDs commuting with $D$ are equivalent to $D$;

\item [\rm (ii)] if $\mathrm{rk}\,D=2$, fix such  a variable $x$ of $B$ that $x\in A=\mathrm{Ker}\,D$. Then there exists $E \in {\rm LND}(B)$, non-equivalent to $D$  such that $DE=ED$ and for any $E^{\prime} \in {\rm LND}(B)$ be such that $DE^{\prime}=E^{\prime}D$, there exist $\alpha_1,\alpha_2 \in k[x]\setminus \{0\}$ and $\beta \in {\rm Ker}~D$ such that $$\alpha_1E^{\prime}=\alpha_2E+\beta D.$$ Moreover ${\rm Ker}\,D \cap {\rm Ker}\,E= {\rm Ker}\,D\cap {\rm Ker}\,E^{\prime}=k[x]$;

\item [\rm (iii)] if $\mathrm{rk}\,D=1$,  then for each variable $x$ of $A$ there exists $E \in {\rm LND}(B)$ non-equivalent to $D$ such that $DE=ED$, ${\rm Ker}\,D \cap {\rm Ker}\,E=k[x]$ and there exists a systen of coordinates $\{x,y,z\}$ in $B$ such that $E$ is triangular with respect to this system. If $E^{\prime} \in {\rm LND}(B)$ be such that $DE^{\prime}=E^{\prime}D$ and ${\rm Ker}\,D \cap {\rm Ker}\,E^{\prime}=k[x]$, then $E^{\prime}$ is triangular with respect to $\{x,y,z\}$ and there exist $\alpha_1,\alpha_2, \in k[x]\setminus \{0\}$ and $\beta \in {\rm Ker}~D$ such that $$\alpha_1E^{\prime}=\alpha_2E+\beta D.$$

\end{enumerate}
\end{thm}

\begin{proof} Since $D$ is irreducible, if there is a non-equivalent $E$ such that $DE=ED$, then by Proposition~\ref{rk2com1}, we can choose $g \in B$ such that $D(g) \in k[x]$, where $x$ is a variable of $B$ contained in ${\rm Ker}~D$. This proves (i).

If ${\rm rk}~D=2$, then such an $E$ exists by Lemma~\ref{rk2lem} and Proposition~\ref{rk2com1}. By Proposition~\ref{frcom}, $\mathrm{Ker}\,D\cap\mathrm{Ker}\,E=k[f]$, where $f$ is a variable of $B$. By Corollary~\ref{var}, $k[f]=k[x]$. We can say the same about $E^{\prime}$. Therefore, $D,E$ and $E^{\prime}$ are $k[x]$-LNDs of $B$. So, we can apply Corollary~\ref{corrin}, which proves (ii). 

If $\mathrm{rk}\,D=1$, then there a coordinate system $(x,y,z)$ of $B$ such that $D=\frac{\partial}{\partial z}$ and $A=k[x,y]$. Clearly, such an $E$ exists by Lemma~\ref{rk2lem} and Proposition~\ref{rk2com1}. As $ED=DE$, $E$ restricts to an LND on $A$. Since $E(x)=0$, by Proposition \ref{frext}, $E(y) \in k[x]$. Again by Proposition \ref{frext}, $E(z) \in A=k[x,y]$. So $E$ is triangular with respect to the coordinate system $(x,y,z)$. Similar argument shows that $E^{\prime}$ is also triangular and the rest of the proof follows from Corollary \ref{corrin}.

\end{proof}

Let us now apply the theory of Section~\ref{cln} to investigate LNDs of rank $2$ of $B=k^{[3]}$. 

\begin{prop}\label{stper}
Let $D$ be an irreducible LND of $B=k^{[3]}$ of rank at most $2$. Suppose $B=k[x,y,z]$ and $x\in A=\mathrm{Ker}\,D$. Let $\mathbf{x}=\{x,y,z\}$ and let $\Gamma$ denote the collection of all coordinate systems $\mathbf{x^{\prime}}=\{x,y^{\prime}, z^{\prime}\}$, where $y^{\prime}, z^{\prime} \in B$.
Then
$$\mathrm{ulev}(D)=\min_{\mathbf{x^{\prime}} \in \Gamma}~\left\{ \mathrm{lev}_{\mathbf{x^{\prime}}}(D)~\vert ~ \mathbf{x^{\prime}} \in \Gamma \right\}.$$
\end{prop}
\begin{proof}
Let $x$ be a variable of $B$ contained in $A$. 
By Corollary~\ref{CLNDC3var},~$D$ can be obtained by a full MC-chain $\frac{\partial}{\partial x}, \frac{\partial}{\partial y},\ldots, D$.  By definition, universal level of $D$ is the minimal level of D with respect to a coordinate system $\{x',y',z'\}$ with $x'\in A$. If $\mathrm{rk}\,D=2$, Corollary~\ref{var} implies $x'=\lambda x+\mu$. So, $\{x,y',z'\}$ is also a system of coordinates and since $k(x')=k(x)$, $\mathrm{lev}_{(x',y',z')}(D)=\mathrm{lev}_{(x,y',z')}(D)$. So, to compute $\mathrm{ulev}\,(D)$ we need to take minimal $\mathrm{lev}_{(x,y',z')}(D)$. If $\mathrm{rk}~D=1$ and $D$ is irreducible, then by Proposition~\ref{fr95c}, $D$ has a slice, say $s$. By Theorem \ref{df}, $A=k[x,f]$, for some $f \in B$. Then, by Theorem \ref{st}, $B=k[x,f,s]$ and $D=\dfrac{\partial}{\partial s}$. Therefore, $\mathrm{lev}_{(x,f,s)}(D)=1$. So, in this case too, $\mathrm{ulev}(D)=\min_{\mathbf{x^{\prime}} \in \Pi}~\left\{ \mathrm{lev}_{\mathbf{x^{\prime}}}(D)~\vert ~ \mathbf{x^{\prime}} \in \Gamma \right\}$.
\end{proof}

\smallskip
\indent
As before, let $D \in {\rm LND}(B)$ be irreducible of rank at most $2$. Let us also fix a coordinate system $\{x,y,z\}$ of $B$ such that $D(x)=0$. Let $\alpha$ be a $k(x)$-automorphism of $k(x)[y,z]$. By~$\theta$, we denote the swapping 
$$
\begin{pmatrix}
y\\
z
\end{pmatrix}\mapsto
\begin{pmatrix}
z\\
y
\end{pmatrix}
$$
\begin{defn}
{\em 
Let us define the {\it complexity of $\alpha$}, denoted by $c(\alpha)$ to be the minimal number~$m$ such that there exist $\tau_0,\tau_1\ldots, \tau_m\in \mathrm{T}_2(k(x))$ with the property that 
$$
\tau_m\circ\theta\circ\tau_{m-1}\circ\ldots\circ\theta\circ\tau_{0}\circ\alpha \in \mathrm{Aut}(k[x,y,z]).
$$
}
\end{defn}
\begin{rem}
{\em
The number $m$ in the previous definition exists. Indeed, it follows from Remark \ref{rr11} that there exists $\eta_0,\eta_1\ldots, \eta_N\in \mathrm{T}_2(k(x))$ that
$$
\eta_N\circ\theta\circ\eta_{N-1}\circ\ldots\circ\theta\circ\eta_{0}\circ\alpha =\mathrm{id}.
$$
}
\end{rem}

\smallskip
\noindent
Let us consider the extension of $D$ as an LND of $k(x)[y,z]$, which we also denote by $D$ and take $\Psi(D)$ as in Section~\ref{sectwovar}.
\begin{thm}\label{comulev}
Suppose $D$ is an irreducible LND of rank 2. Then 
$${\rm ulev}(D)=c(\alpha)+2,~~\text{where}~~\alpha\in\Psi(D).$$
\end{thm}
\begin{proof}
 Let us recall that $\alpha\in\Psi(D)$ takes $y$ to a generator of $\mathrm{Ker}\,D$ and $z$ to a slice of $D$. Since the mapping $\Psi$ depends on the coordinate system, let us write $\alpha \in\Psi_{(x,y,z)}(D)$.
If we take another coordinate system $\{x,y',z'\}$, then we can consider the automorphism $\varphi$ of $B$ given by $(x,y',z')\mapsto(x,y,z)$. In this system of variables we also can consider $\Psi_{(x,y',z')}(D)$. It is easy to see that $\alpha\circ \varphi\in \Psi_{(x,y',z')}(D)$. Let us also recall that $\mathrm{lev}_{(x,y,z)}(D)=m+2$, where $m$ is the number of affine automorphisms in the decomposition of $\alpha=\tau_m\circ\gamma_m\ldots\circ\tau_1\circ\gamma_1$, $\tau_i\in \mathrm{T}_2(k(x))$, $\gamma_i\in\mathrm{Af}_2(k(x))$ ($1 \leqslant i \leqslant m$), see~Remark~\ref{rr11}. Equivalently, $\mathrm{lev}_{(x,y,z)}(D)=m+2$, where $m$ is the number of swappings  $\theta$ in the decomposition $\alpha=\tau_m\circ\theta\ldots\circ\tau_1\circ\theta\circ\tau_0$, $\tau_j\in \mathrm{T}_2(k(x))$ ($0 \leqslant j \leqslant m$). By Proposition~\ref{stper}, ${\rm ulev}(D)$ is the minimal $\mathrm{lev}_{(x,y^{\prime},z^{\prime})}(D)$, where $k[x,y,z]=k[x,y^{\prime},z^{\prime}]$. It follows from the definition, that $c(\alpha)=c(\alpha \circ \varphi)$ for any $\varphi \in {\rm Aut}(B)$. Thus if we take $m$ to be the minimal number of swappings $\theta$ in the decomposition of $\alpha \circ  \varphi$ for all $\varphi\in \mathrm{Aut}(B)$, such that $\varphi(x)=x$, then ${\rm ulev}(D)=m+2$. The result follows.
\end{proof}

\begin{rem}\label{comr1}
{\em If $\mathrm{rk}\,D=1$, then $\mathrm{ulev}\,D=1$ and the complexity of $\alpha$ equals zero.
}
\end{rem}

\smallskip
\noindent
The following corollary gives a criterion for an irreducible LND of rank at most $2$ to be triangularizable.
\begin{cor}\label{ustr}
Let $D$ be an irreducible LND of $k[x,y,z]$ such that $D(x)=0$. Consider $\alpha\in\Psi(D)$. Then $$D~\text{ is triangularizable}~\Leftrightarrow ~c (\alpha) \leqslant 1.$$
Moreover, if $D$ is triangularizable, then there exist $u_1,v_0 \in k(x)$, $f,g \in k[x,y,z]$, $F_0 \in k(x)^{[1]}$ and $\zeta \in {\rm Aut}_{k[x]}(k[x,y,z])$ such that 
$$\zeta(y)=u_1(v_0g+F_0(f)).$$ 
\end{cor}

\begin{proof} The proof follows from Theorem~\ref{crtr}, Theorem~\ref{comulev} and Remark~\ref{comr1}. More explicitly,~$D$ is triangularizable if and only if there exists a decomposition
$$\varphi=\tau_1\circ\theta\circ\tau_0\circ\alpha, $$ 
where $\varphi \in {\rm Aut}(k[x,y,z])$ and $\tau_0,\tau_1 \in {\rm T}_2(k(x))$.
Denote $\alpha(y)=f(y,z)$, $\alpha(z)=g(y,z)$, where $f,g\in k(x)[y,z]$. Let 
$$\tau_0\colon
\begin{pmatrix}
y\\
z
\end{pmatrix}
\rightarrow
\begin{pmatrix}
u_0y\\
v_0z+F_0(y)
\end{pmatrix}, \qquad
\tau_1\colon
\begin{pmatrix}
y\\
z
\end{pmatrix}
\rightarrow
\begin{pmatrix}
u_1y\\
v_1z+F_1(y)
\end{pmatrix}.
$$
Here $u_i,v_i\in k(x)$, $F_i\in k(x)^{[1]}$ for $i=0,1$.
Then
$$
\tau_1\circ\theta\circ\tau_0\circ\alpha\colon
\begin{pmatrix}
y\\
z
\end{pmatrix}
\rightarrow
\begin{pmatrix}
u_1(v_0g+F_0(f))\\
v_1u_0f+F_1\left(v_0g+F_0(f)\right)
\end{pmatrix}.
$$
If $D$ is triangularizable, then there exists an $k[x]$-automorphism of $k[x,y,z]$ of this form. By suitably choosing $u_1$ and $v_0$, one can further assume $f,g \in k[x,y,z]$ in the expression $u_1(v_0g+F_0(f))$.
\end{proof}

\smallskip
\indent
It turns out that in many cases, the polynomial 
$u_1(v_0g+F_0(f))$ cannot be the image of~$y$ under any $k[x]$-automorphism of $k[x,y,z]$. Let us fix two non-negative integers $p$ and~$q$ not both zero. Let us consider the $\mathbb{Z}$-grading on $k(x)[y,z]$ given by 
$\deg y=p$,  $\deg z=q$. If $h\in k(x)[y,z]$, we denote by $\overline{h}$ the sum of terms with the biggest degree with respect to this grading. Since $\overline{h}$ is homogeneous, by $\deg(\overline{h})$, we mean the degree of any its terms. We denote the sum of linear terms of $h$ by $L(h)$.

\begin{defn}
{\em
Let $h\in k[x,y,z]$. Denote the greatest common divisor of all coefficients in $k[x]$  of this polynomial, considered as polynomial in $y$ and $z$, by $d(h)$. We say that a polynomial $h$ is {\it primitive} if $d(h)=1$. 
}
\end{defn}

\begin{prop}\label{usnap} Let $f,g\in k[x,y,z]$. Suppose, $\gcd(d(L(f)),d(L(g)))\in k[x] \setminus k$. Suppose, $g$ is a  primitive polynomial such that $g(x,0,0)=0$. Assume there exist non-negative integers $p$ and $q$ such that $(p,q) \neq (0,0)$,  $\overline{f}$ is primitive and 
$\deg(\overline{f})>\deg(\overline{g})$. Then there does not exist any $k[x]$-automorphism $\zeta$ of $k[x,y,z]$, such that  $\zeta(y)=u_1(v_0g+F_0(f))$ for any $u_1,v_0\in k(x)$ and $F_0\in k(x)[t]=k(x)^{[1]}$.
\end{prop}
\begin{proof} Suppose that, $\zeta$ is a $k[x]$-automorphism of $k[x,y,z]$, such that  $\zeta(y)=u_1(v_0g+F_0(f))$ for some $u_1,v_0\in k(x)$ and $F_0\in k(x)[t]$. Denote
$$u_1F_0[t]=a_nt^n+\ldots +a_0, \qquad \text{where}~a_i\in k(x), ~0 \leqslant i \leqslant n.$$
Let us show that $a_i\in k[x]$ for all $i\geqslant 0$. If it's not the case, we can define $$m :=max ~\{i ~\vert ~ a_i \notin k[x]\}.$$ Then terms of degree $m\deg(\overline{f})$ in $\zeta(y)$ coincide with terms of degree $m\deg(\overline{f})$ in $u_1F_0(f)$. Since for all $j > m$, $a_j \in k[x]$ and $\overline{f}$ is primitive, sum of terms of degree $m\deg(\overline{f})$ in $u_1F_0(f)$ does not lie in $k[x,y,z]$. Therefore, all $a_i$ are in $k[x]$. So  $u_1v_0g$ in $k[x,y,z]$. Since $g$ is primitive, $u_1v_0\in k[x]$. Hence, $\gcd(d(L(a_1f)),d(L(u_1v_0g)))\notin k$. As a result, the ideal $$\Big (\dfrac{\partial \zeta(y)}{\partial y}, \dfrac{\partial \zeta(y)}{\partial z} \Big) \triangleleft k[x,y,z]$$
is proper. So we have a contradiction.
\end{proof}

\smallskip
\noindent
\begin{ex}\label{ex1} 
{\em
Let us fix non-zero polynomials $r_1,r_2\in k[x]$ and $h_1,h_2\in k[x,t]$ such that 
\begin{itemize}
\item $\gcd (r_1,r_2)=r\in k[x]\setminus k$;
\item $h_i=a_it^{m_i}+b_{i,m_i-1}(x)t^{m_i-1}+\ldots+b_{i,2}(x)t^2$, $i=1,2$ and $2 \leqslant j \leqslant m_i-1$, where $a_i\in k\setminus\{0\}$ and $b_{i,j} \in k[x]$.
\end{itemize}
We consider $h_i=h_i(t)$ as a polynomial in one variable $t$ with coefficients in $k[x]$ and denote its formal derivative by $h_i'=h_i'(t)$.

 \smallskip
\noindent
Consider the following automorphisms from $\mathrm{T}_2(k(x))$:
$$
\beta_1\colon
\begin{pmatrix}
y\\
z
\end{pmatrix}
\rightarrow
\begin{pmatrix}
y\\
r_1(x)z-h_1(y)
\end{pmatrix}\text{ and }
\beta_2\colon
\begin{pmatrix}
y\\
z
\end{pmatrix}
\rightarrow
\begin{pmatrix}
y\\
r_2(x)z-h_2(y)
\end{pmatrix}.
$$
Let us denote $\sigma=r_1(x)z-h_1(y)$.
Then we consider the composition 
$$\alpha=\theta\circ\beta_2\circ\theta\circ\beta_1\colon
\begin{pmatrix}
y\\
z
\end{pmatrix}
\rightarrow
\begin{pmatrix}
r_2(x)y-h_2(\sigma)\\
\\
\sigma
\end{pmatrix}.
$$
Set  $f:=r_2(x)y-h_2(\sigma)=r_2(x)y-h_2(r_1(x)z-h_1(y))$ and $g:=\sigma=r_1(x)z-h_1(y)$. 
Let $D$ be the corresponding $k[x]$-LND of $B$. 
According to this decomposition, we obtain a full MC-chain 
\begin{multline*}
\delta_1=\frac{\partial}{\partial y},\qquad \delta_2=\frac{\partial}{\partial z},\qquad \delta_3\colon
\begin{pmatrix}
y\\
z
\end{pmatrix}
\rightarrow
\begin{pmatrix}
r_1(x)\\
h_1'(y)
\end{pmatrix},
\\
\delta_4=D\colon\begin{pmatrix}
y\\
z
\end{pmatrix}
\rightarrow
\begin{pmatrix}
h_2'(\sigma)r_1(x)\\
r_2(x)+h_2'(\sigma)h_1'(y)
\end{pmatrix}.
\end{multline*}
So, $\mathrm{lev}_{(x,y,z)}D=4$ and one can write $D$ explicitly as
$$D=\Big{(}h_2'(r_1(x)z-h_1(y))r_1(x)\Big{)}\dfrac{\partial}{\partial y}+\Big{(}r_2(x)+h_2'(r_1(x)z-h_1(y))h_1'(y)\Big{)}\dfrac{\partial }{\partial z}.$$
 Let us prove that $\mathrm{ulev}\,D=4$, i.e. that $D$ is non-triangularizable. 
Take $p=1,q=0$. Then
$\overline{f}=\overline{h_2(h_1(y))}=a_1a_2y^{m_1m_2}$, $\overline{g}=\overline{h_1(y)}=a_1y^{m_1}$. Hence, $\deg \overline{f}>\deg \overline{g}$ and polynomials $\overline{f}$ and~$g$ are primitive. We have $L(f)=r_2(x)y$, $L(g)=r_1(x)z$. Therefore, 
$\gcd(d(L(f),d(L(g)))=r(x)\in k[x]\setminus k$. Also we have $g(x,0,0)=h_1(0)=0$. So, all conditions of Proposition~\ref{usnap} are satisfied. Hence, by Corollary~\ref{ustr}, $D$ is non-triangularizable.

}
\end{ex}

\begin{ex}\label{ex2} 
{\em
Following the notations of Example~\ref{ex1}, let us take  $r_1(x)=r_2(x)=x$, $h_1(x,t)=h_2(x,t)=t^2$. Then 
$$D=2x(xz-y^2)\dfrac{\partial}{\partial y}+\Big {(}x+4y(xz-y^2)\Big{)}\dfrac{\partial }{\partial z}$$
 is non-triangularizable.
}
\end{ex}

\begin{rem}\label{varema}
{\em
When we put  $r_1(x)=r_2(x)=x$, $h_1(x,t)=h_2(x,t)=t^2$ as in Example~\ref{ex2}, the automorphism $\alpha$ from Example~\ref{ex1} equals $\alpha=\theta\circ\beta\circ\theta\circ\beta,$ where 
$$
\beta\colon
\begin{pmatrix}
y\\
z
\end{pmatrix}
\rightarrow
\begin{pmatrix}
y\\
xz-y^2
\end{pmatrix}.
$$
If we would like to construct LND with universal level $5$, the complexity of the corresponding automorphism should be $3$. The natural candidate is
$\varphi=\theta\circ\beta\circ\theta\circ\beta\circ\theta\circ\beta.$
}
\end{rem}

\smallskip
\noindent
This leads to the following question. An affirmative answer to this question will answer Question~\ref{qu} affirmatively.
\begin{que}\label{qq}
Is $c(\varphi)=3$?
\end{que}

\section{Rank $\mathbf{3}$ derivations}\label{rk3der}
In~\cite{F98}, Freudenburg introduced the following  example of a $\mathrm{rank}\,3$ derivation $\Delta$ of $B=k[x,y,z]$. Denote 
$$u=xz-y^2~~\text{and}~~ v=zu^2+2x^2yu+x^5.$$
Then $\Delta={\rm Jac}(u,v,\cdot).$
The proof that $\Delta$ has rank 3 is based on the fact that $\Delta$ is homogeneous with respect to the standard $\mathbb{Z}$-grading. Till now, all known rank $3$  derivations are $\mathbb{Z}$-homogeneous for suitable $\mathbb{Z}$-gradings.  In this section, we give a new proof of the fact that $\Delta$ has rank 3. This proof does not use the homogeneity of $\Delta$.  Moreover, we prove a sufficient condition for an LND obtained by local slice construction to be of rank~3. This approach allows us to give new examples of rank $3$ LNDs, which are not homogeneous. In particular, we construct a family of rank $3$ LNDs, of which $\Delta$ is a member.

\smallskip
\noindent
We begin with an elementary lemma.
\begin{lem}\label{rk3com1}
Let $B=k^{[3]}$, $D (\neq 0) \in {\rm LND}(B)$ and $A:={\rm Ker}~D$. Let $g$ be a local slice of $D$ and let $D(g)=f\ \in A$. Then $A[g]$ is $D$-invariant and the restriction of $D$ to $A[g]$ equals $f\dfrac{\partial }{\partial g}\in {\rm LND}(A[g])$. 
\end{lem}

\begin{proof} Let $A=k[u,v](=k^{[2]})$. Let $\delta =f\dfrac{\partial }{\partial g}$. Then $\delta(u)=\delta(v)=0$ and $\delta(g)=f$. So $D\vert_{A[g]}=\delta \in {\rm LND}(A[g])$. Since $g$ is a local slice of $D$, we have $k[x,y,z,\frac{1}{f}]=k[u,v,\frac{1}{f}][g]$. So any element $s \in B$ can be represented as 
$$s=\dfrac{f_s(u,v,g)}{f^{m_s}},~\text{where}~f_s \in k[u,v,g]=A[g]~\text{and}~m_s \in \bN \cup \{0\}.$$
One can extend $\delta$ to $B$ by defining $\delta(s):=\dfrac{\delta(f_s)}{f^{m_s}}$. Then $$f^{m_s}{\delta}(s)=\delta(f_s(u,v,g))=f\dfrac{\partial f_s}{\partial g}=f^{m_s}D (s).$$ So $\delta=D$ in $B$.
\end{proof}

\smallskip
\noindent
\begin{rem}\label{nepsam}
{\em In notations as above, $D$ admits a non-equivalent LND $E$ of $B$ commuting with $D$ if and only if $\delta=D|_{A[g]}$ admits a non-equivalent LND $\xi$ of $A[g]$.
Indeed, $f$ generates both plinth ideals $\mathrm{pl}(D)$ and  $\mathrm{pl}(\delta)$. Then Proposition~\ref{rk2com1} implies the goal.
}
\end{rem}

\smallskip
\noindent
Thus we have the following result.
\begin{prop}\label{rk3com2}
Let $B=k^{[3]}$, $D (\neq 0) \in {\rm LND}(B)$ be of rank $3$ and $A:={\rm Ker}~D$. Then~if $g$ is any minimal local slice of $D$, then $D(g)$ cannot be written as an univariate polynomial in any variable of $A$.
\end{prop}
\begin{proof}Suppose the converse. Then there exists $x \in A$ such that $A=k[x]^{[1]}$ and $D(g) \in k[x]$. By Lemma \ref{dtriang} (ii), $x$ is a variable in $B$ and since $\rm{rk}~D=3$, we have a contradiction. 
\end{proof}

Lemma~\ref{rk3com1} and Proposition~\ref{rk3com2} imply that $D$ is an extension of rank 1 derivation $f\frac{\partial}{\partial g}$ on $A[g]$, where $f=D(g)\in A$ cannot be written as an univariate polynomial in any variable of~$A$. Note that similar extension we obtain considering any slice $h$ of $D$. Indeed, since $D(h)$ is divisible by $f=D(g)$, it cannot be written as an univariate polynomial in any variable of~$A$. 
If $D$ is irreducible, the converse of Proposition \ref{rk3com2} is true. So we obtain the following statement.
\begin{prop}\label{rrr} Let $B=k^{[3]}$ and $D$ be an irreducible LND of $B$. Put $A=\mathrm{Ker}\,D$. Let $g$ be a minimal local slice of $D$. Suppose $D(g)$ cannot be written as an univariate polynomial in any variable of $A$. Then $\mathrm{rk}\,D=3$.
\end{prop}
\begin{proof} Assume that $D$ is irreducible of rank at most $2$, by Lemma \ref{rk2lem}, there exists $\tilde{g} \in B$ and a variable $x$ of $B$ in $A$ such that $D(\tilde{g}) \in k[x]$. Let $g$ be any minimal local slice of $D$. Then $D(g) \mid D(\tilde{g})$. Hence $D(g)$ is a univariate polynomial in $x$, which by Lemma \ref{dtriang} (i) is also a variable in $A$.
\end{proof}

Using the assertions above we obtain the following criterion for an irreducible derivation $D$ of $k^{[3]}$ to be an LND of rank 3.

\begin{thm}\label{usrt}
Let $B=k[x,y,z]$ and $D$ be an irreducible derivation of $B$. Then $D$ is an LND of rank 3 if and only if there are polynomials $f,h,r\in B$ such that 
\begin{enumerate}
\item [\rm (i)] $\mathrm{Ker}\,D=k[f,h]$;

\item [\rm (ii)] $D(r)=v(f,h)$, where $v(f,h)\in k[f,h]$ is not an univariate polynomial in $k[f,h]$;

\item [\rm (iii)] there are no such $p,q,l\in k[f,h]$ and $s\in B$, that $r-p=sq$, $v=ql$, and  $q\notin k$;

\item [\rm (iv)] $x,y,z\in k[f,h,r,\frac{1}{v}]$.
\end{enumerate}
\end{thm}
\begin{proof}
Suppose that the conditions (i)-(iv) hold. Then $D$ is an LND. Indeed, $D(f)=D(h)=D(v)=0$ and $D(r)\in k[f,h]$. So $D$ is locally nilpotent on $k[x,y,z,\frac{1}{v}]=k[f,h,r,\frac{1}{v}]\supseteq B$. Hence, $D$ is an LND of $B$.
Let $A=\mathrm{Ker}\,D= k[f,h]$ and $l\in A$ be the generator of the plinth ideal $\mathrm{pl}(D)$. Then there exists $s\in B$, that $D(s)=l$. Since $v\in \mathrm{pl}(D)=(l)$, there exists $q\in A$, such that $v=ql$. Therefore, $D(r-sq)=v-lq=0$. So, $r-sq=p\in A$. By (iii), $q\in k^*$. Hence, $l=\dfrac{v}{q}$ is not a univariate polynomial. So there is no univariate polynomial in $\mathrm{pl}(D)=A \cap {\rm Im}~D$. By Proposition~\ref{rrr}, $\mathrm{rk}\,D=3$.

Conversely, suppose $D$ is locally nilpotent and $\mathrm{rk}\, D=3$. Then there exist $f,h,r, v \in B$ such that $A=\mathrm{Ker}\,D=k[f,h]$, $\mathrm{pl}(D)=(v)$, and $D(r)=v$. By Proposition~\ref{rk3com2}, $v$ is not univariate. If there exist $p,q,l\in k[f,h]$ and $s\in B$, such that $r-p=sq$, $v=ql$ and  $q\notin k$, then $D(s)=D\left(\frac{r-p}{q}\right)=l\notin (v)$. This contradicts that $l \in \mathrm{pl}(D)=(v)$.  By Proposition~\ref{proplocsl},  $x,y,z\in k[f,h,r,\frac{1}{v}]$. 
\end{proof}

\smallskip
\noindent
The following corollary gives sufficient conditions for a LND of $k^{[3]}$ obtained by a local slice construction to be of rank $3$.
\begin{cor}\label{lscor}
Let $D$ be an irreducible LND of $B=k[x,y,z]$ with ${\rm Ker}~D=k[f,g]$. Let $E:=Jac(f,h,.)$ be the derivation obtained from $D$ by a local slice construction using the data $(f,g,r)$, where $D(r)=gP(f)\neq 0$ and $P \in k^{[1]}$ is not a constant. If $E$ is irreducible and $r$ is a minimal local slice for $E$ then $E$ is an LND of rank $3$.
\end{cor}

\begin{proof} By Theorem~\ref{locslcon} the derivation $E$ is an irreducible LND of $B$ with ${\rm Ker}~E=k[f,h]$ and $E(r)=-hP(f)$. Suppose there exist $\alpha ,\beta, l \in k[f,h]$ and $s \in B$ such that $r-\alpha=s\beta$ and $\beta l=-hP(f)$. Then $E(s)=E(\frac{r-\alpha}{\beta})=l \in {\rm pl}(E)=(-hP(f)).$ So $\beta \in k^*$. Since $E$ satisfies all the conditions (i)-(iv) of Theorem \ref{usrt}, we have $E$ to be of rank $3$.
\end{proof}

\medskip
\indent
Now, let us introduce a family of derivations. Let us fix integer numbers $m\geqslant 2$, $n\geqslant 1$ and a polynomial $F\in k[t_1,t_2]$ such that $F(t_1,0) \notin k$.

Let us consider an LND $D=x\frac{\partial}{\partial y}+my^{m-1}\frac{\partial}{\partial z}$. Let us put 
$$f=xz-y^m,\qquad g=x,\qquad r=xF(x,f)+yf^n.$$ Then $\mathrm{Ker}\,D=k[f,g]$ and $D(r)=gf^n$. 
It is easy to see that the polynomial $\phi\in k[f]^{[1]}$ of minimal degree such that $g\mid \phi(r)$ equals $\phi(f,r)=f^{mn+1}+r^m$. So we have $h=\frac{f^{mn+1}+r^m}{x}$ and $E=E(m,n,F)=Jac(f,h,\cdot)$ is the derivation obtained from $D$ by a local slice construction using the data $(f,g,r)$.
We have $E(2,1,t_1^2)=\Delta$. Indeed, in this case $f=xz-y^2$, $r=x^3+yf$, 
$h=\frac{f^3+r^2}{x}=f^2z+x(x^2)^2+2x^2yf=zf^2+2x^2yf+x^5$. So, $\{E(m,n,F)\}$ is a family of derivations containing $\Delta$. It is interesting to note that the only homogeneous derivations with respect to standard $\mathbb{Z}$-grading among $E(m,n,F)$ are $E(2,n,S(t_1^2, t_2))$, where $S$ is a homogeneous polynomial in two variables of degree $n$.
\begin{thm}\label{rk3final}
Let $B=k[x,y,z]$. For each $m\geqslant 2$, $n\geqslant 1$ and $F\in k[t_1,t_2]$ such that $F(t_1,0)\notin k$, the derivation $E=E(m,n,F)$ is an LND of $B$ with ${\rm rk}~E=3$.
\end{thm}
\begin{proof}
To prove the goal applying Corollary~\ref{lscor} we need to prove that $E$ is irreducible and $r$ is a minimal local slice for $E$. 

{\it Irreducibility of E.} Suppose that the image of $E$ is contained in a proper ideal $(a)$, where $a \in B$.  Since $E(r)=-hf^n$, we have $a\mid hf^n$. 
We have
$$
E(x)=E\left(\frac{f^{mn+1}+r^m}{h} \right)=\frac{mr^{m-1}(-f^nh)}{h}=-mr^{m-1}f^n.
$$
Since
$$
y=\frac{r-xF(x,f)}{f^n},
$$
we have
$$
E(y)=E\left(\frac{r-xF(x,f)}{f^n} \right)=-r+\frac{\partial (xF(x,f))}{\partial x}(mr^{m-1}).
$$
By definition
\begin{multline*}
h=\frac{f^{mn+1}+r^m}{x}=\frac{f^{mn}(xz-y^m)+(xF(x,f)+yf^n)^m}{x}=\\
=f^{nm}z+x^{m-1}F^{m}(x,f)+mx^{m-2}F^{m-1}(x,f)yf^n+\ldots+mF(x,f)y^{m-1}f^{mn-n}.
\end{multline*}
Therefore, $h-x^{m-1}F^{m}(x,0)$ is divisible by $f$. Also $r-xF(x,0)$  is divisible by $f$.
Hence, there exists $u\in B$ such that
\begin{multline*}
E(y)=-x^{m-1}F^{m}(x,0)+\frac{\partial (xF(x,0))}{\partial x}(mx^{m-1}F(x,0)^{m-1})+fu=\\
=(m-1)x^{m-1}F^{m}(x,0)+mx^mF(x,0)^{m-1}\frac{\partial (F(x,0))}{\partial x}+fu.
\end{multline*}
A polynomial in $x$ is divisible by $f$ if and only if it is zero. Let $\alpha x^k$  be the leading term of $F(x,0)$. Then the coefficient in 
$(m-1)x^{m-1}F^{m}(x,0)+mx^mF(x,0)^{m-1}\frac{\partial (F(x,0))}{\partial x}$ at $x^{mk+m-1}$ is equal to $(m-1)\alpha^m+mk\alpha^m=(mk+m-1)\alpha^m\neq 0$.
Therefore, $f\nmid D(y)$. It implies $f\nmid a$, and hence, $a\mid h$.

Since $E(x)=-mr^{m-1}f^n$, we obtain $a\mid r^{m-1}f^n$. Since $f$ is irreducible and $f\nmid a$, we have $a\mid r^{m-1}$. So, we have $a\mid h$ and $a\mid  r^{m-1}$. This implies $\gcd (h,r)\notin k^*$. Then 
$$\gcd(xh,r)=\gcd(f^{mn+1}+r^m,r)=\gcd(f^{mn+1},r) \notin k^*.$$
 
Therefore, $f\mid r$. But $r=xF(x,f)+yf^n$. If $f\mid r$, then $f\mid xF(x,0)$. But $F(x,0)\notin k$, hence $xF(x,0)$ is a non-zero polynomial in $x$. Since $f$ depends on $y$ and $z$, it cannot divide $xF(x,0)$.

 Thus we have a contradiction and hence  $E$ is irreducible.

{\it Minimality of $r$.}
Denote $v=-hf^n=E(r)$. Assume there exist such $s\in B$ that $E(s)=\frac{v}{q}$ for some $q\in k[f,h]\setminus k$. Then $E(r-sq)=0$, i.e. $r-sq=p\in k[f,h]$. We have $q=f^ch^d$, where $0\leq c\leq n$, $0\leq d\leq 1$. 
Suppose $d=0$. Then $f\mid (r-p)$ i.e. $f\mid (xF(x,f)+yf^n-p)$. Since $p\in k[f,h]$, we write $p=p(f,h)$. Then $f\mid (xF(x,0)-p(0,h))$. But 
\begin{multline*}
h=f^{mn}z+x^{m-1}F^{m}(x,f)+mx^{m-2}F^{m-1}(x,f)yf^n+\ldots+mF(x,f)y^{m-1}f^{mn-n}=\\
=x^{m-1}F^{m}(x,0)+ft, \text{ for some t}\in B.
\end{multline*}
So, we have $f\mid (xF(x,0)-G(x^{m-1}F^{m}(x,0))$ for some univariate polynomial $G$. But 
$$xF(x,0)-G(x^{m-1}F^{m}(x,0))$$ 
is a non-zero polynomial in $x$, because $m \geqslant 2$ and $F(x,0) \notin k$.
This is a contradiction as non-zero polynomial in $x$ cannot be divisible by~$f$.  Therefore, $r-p$ is not divisible by~$f$. So we can assume $d=1$. 
Then $h \mid (r-p)$ i.e. $h\mid (xF(x,f)+yf^n-p)$. Hence $h\mid (xF(x,f)+yf^n-p(f,0))$. Putting $y=z=0$,  we have $f(x,0,0)=0,~~h(x,0,0)=x^{m-1}F^{m}(x,0)~~\text{and}~~r(x,0,0)=xF(x,0)$. Then 
$$x^{m-1}F^m(x,0) \mid (xF(x,0)-p(0,x^{m-1}F^m(x,0))$$ which is again a contradiction as $m\geqslant 2$ and $F(x,0) \notin k$. So $r-p$ is not divisible by $h$. Hence, $r$ is a minimal local slise for $E$.

Thus, all conditions of Corollary~\ref{lscor} are satisfied. Therefore, $E(m,k,F)$ is an LND of rank 3. 

\end{proof}

So we see that Freudenburg's example can be considered as a member of a large family of rank $3$ LNDs on $k^{[3]}$.

\medskip

\begin{center}
{\bf Acknowledgement:}
\end{center}

The authors are grateful to Professor Ivan Arzhantsev for his valuable comments while going through the earlier drafts and suggesting improvements and to Timofey Vilkin for fruitfull discussions. The authors were supported by Indo Russia Project DST/INT/RUS/RSF/P-48/2021 with TPN 64842 and RSF grant no. 22-41-02019. The second author is a Young Russian Mathematics award winner and would like to thank its sponsors and jury.

	{\small{

}}

\end{document}